\documentclass[11pt, twoside]{article}
\pdfoutput=1

\usepackage{graphicx}
\usepackage[caption=false]{subfig}
\usepackage{amsmath}
\usepackage{amssymb,amsfonts}
\usepackage{amsthm}
\usepackage{bm}
\usepackage{mathrsfs}
\usepackage{amssymb}
\usepackage{multirow}

\newcommand{\norm}[1]{\left\lVert#1\right\rVert}
\newcommand{\innOm}[2]{(#1,#2)_\Omega}
\newcommand{\innGa}[2]{\langle #1,#2\rangle_\Gamma}

\usepackage[margin=1in]{geometry}

\usepackage{algorithm}
\usepackage{algorithmic}

\numberwithin{equation}{section}

\theoremstyle{definition}
\newtheorem{theorem}{Theorem}

\newtheorem{proposition}{Proposition}
\newtheorem{lemma}{Lemma}
\newtheorem{corollary}{Corollary}

\newtheorem{remark}{Remark}
\newtheorem{example}{Example}

\usepackage{cite}
\usepackage{hyperref}
\usepackage[nameinlink]{cleveref}

\usepackage{fancyhdr}
\begin{document}
	
	\title{A Class of Embedded DG Methods for Dirichlet Boundary Control of Convection Diffusion PDEs}

\author{Gang Chen%
	\thanks{School of Mathematics Sciences, University of Electronic Science and Technology of China, Chengdu, China (\mbox{cglwdm@uestc.edu.cn}).}
	\and
	Guosheng Fu
	\thanks{Division of Applied Mathematics, Brown University, RI, USA(\mbox{guoshengfu@brown.edu}).}
	\and
	John~R.~Singler%
	\thanks{Department of Mathematics and Statistics, Missouri University of Science and Technology, Rolla, MO, USA (\mbox{singlerj@mst.edu}.}
	\and
	Yangwen Zhang%
	\thanks{Department of Mathematics Science, University of Delaware, Nework, DE, USA (\mbox{ywzhangf@udel.edu}.}
}

\date{\today}

\maketitle

\begin{abstract}
	We investigated an hybridizable discontinuous Galerkin (HDG) method for a convection diffusion Dirichlet boundary control problem in our earlier work [SIAM J. Numer. Anal. 56 (2018) 2262-2287] and obtained an optimal convergence rate for the control under some assumptions on the desired state and the domain. In this work, we obtain the same convergence rate for the control using a class of embedded DG methods proposed by Nguyen, Peraire and Cockburn [J. Comput. Phys. vol. 302 (2015), pp. 674-692] for simulating fluid flows. Since the global system for  embedded DG methods uses continuous elements,  the number of degrees of freedom for the embedded DG methods are smaller than the HDG method, which uses discontinuous elements for the global system. Moreover, we introduce a new simpler numerical analysis technique to handle low regularity solutions of the boundary control problem. We present some numerical experiments to confirm our theoretical results.
\end{abstract}

\section{Introduction}
\label{intro}
We {study} the following Dirichlet boundary control problem: Minimize the cost functional
\begin{align}
\min J(u)=\frac{1}{2}\| y- y_{d}\|^2_{L^{2}(\Omega)}+\frac{\gamma}{2}\|u\|^2_{L^{2}(\Gamma)}, \quad \gamma>0, \label{cost1}
\end{align}
subject to
\begin{equation}\label{Ori_problem}
\begin{split}
-\varepsilon\Delta y + \bm \beta \cdot \nabla y&=f \quad\text{in}~\Omega,\\
y&=u\quad\text{on}~\partial\Omega,
\end{split}
\end{equation}
where  $\Omega\subset \mathbb{R}^{d}$ $ (d\geq 2)$ is a Lipschitz polyhedral domain with boundary $\Gamma = \partial \Omega$.  In the 2D case,  the optimal control problem \eqref{cost1}-\eqref{Ori_problem} has been proven in  \cite{HuMateosSinglerZhangZhang1} to be  equivalent to the following optimality system
\begin{subequations}\label{eq_adeq}
	\begin{align}
	-\varepsilon\Delta y  + \bm \beta \cdot \nabla y&=f\qquad\qquad\text{in}~\Omega,\label{eq_adeq_a}\\
	y&=u\qquad\qquad\text{on}~\partial\Omega,\label{eq_adeq_b}\\
	-\varepsilon\Delta z - \nabla\cdot(\bm \beta z) &=y-y_d\qquad\text{in}~\Omega,\label{eq_adeq_c}\\
	z&=0\qquad\qquad~\text{on}~\partial\Omega,\label{eq_adeq_d}\\
	\varepsilon\partial_{\bm n} z  -\gamma  u&=0\qquad\qquad~\text{on}~\partial\Omega.\label{eq_adeq_e}
	\end{align}
\end{subequations}

Dirichlet boundary control plays an important role in many applications; see, e.g., \cite{John_Wachsmuth_DBC_NFAO_2009,You_Ding_Zhou_DBC_SICON_1998,Gunzburger_Hou_Svobodny_NS_M2AN_1991,Gunzburger_Hou_Svobodny_Drag_SICON_1992,Hou_Svobodny_NS_JMAA_1993} for flow control problems. Approximating the solution of a Dirichlet boundary control problem can be very difficult since solutions frequently have low regularity. Rigorous convergence results have only been recently obtained for Dirichlet boundary control for the Poisson equation using the continuous Galerkin (CG) method \cite{Casas_Gunther_Mateos_Paradox_SICON_2011,Casas_Raymond_First_SICON_2006,Deckelnick_Andreas_Hinze_Three_SICON_2009,Kunisch_Vexler_Constrained_SICON_2007,May_Rannacher_Vexler_High_SICON_2013,Mateos_Neitzel_Constrained_COA_2016,APel_Mateos_Pfefferer_Arnd_DBC_MCRF_2018} and a mixed finite element method \cite{Gong_Yan_Mixed_SICON_2011}. A potential weakness of the CG method is that the control and state spaces are coupled: the control space is the trace of the state space. A mixed method allows the control and state spaces to be decoupled, which provides greater flexibility compared to the CG method; however,  the degrees of freedom are larger than the CG scheme. It is worth mentioning that Apel et al.\ in \cite{APel_Mateos_Pfefferer_Arnd_DBC_MCRF_2018} is the first work to obtain a superlinear convergence rate for the control on convex polygonal domains if one uses a superconvergence mesh.

Recently, researchers have investigated discontinuous Galerkin (DG) methods for Dirichlet boundary control problems. We used a hybridizable discontinuous Galerkin (HDG) method for the Poisson equation in \cite{HuShenSinglerZhangZheng_HDG_Dirichlet_control1}, and obtained a superlinear convergence rate for the control without using a special mesh or a higher order element. More recently, convection diffusion Dirichlet boundary control problems have gained more and more attention. Benner et al.\ in \cite{PeterBenner} used a local discontinuous Galerkin (LDG) method to obtain a sublinear convergence rate for the control.  We considered an HDG method and proved optimal  superlinear convergence rate for the control in \cite{HuMateosSinglerZhangZhang1}  if the regularity of the solution is high, i.e., $y \in H^{1+s}(\Omega)$ with $s\ge 1/2$. To overcome the difficulty for the low regularity case ($0 \le s<1/2$), we utilized a special projection operator  to get an optimal  superlinear convergence rate in \cite{HuMateosSinglerZhangZhang2}; the numerical analysis was more complicated than in \cite{HuMateosSinglerZhangZhang1}. Furthermore, in contrast to  \cite{HuShenSinglerZhangZheng_HDG_Dirichlet_control1}, we obtained an optimal superliner convergence  rate  for the control by using a discontinuous higher order (quadratic) element. 

Although the degrees of freedom of the HDG method are significantly reduced compared to standard mixed methods,  DG methods and LDG methods, they are still larger than the degrees of freedom of the CG method. In this work, we use  embedded DG (EDG) and interior EDG (IEDG) methods to approximate the solution of the Dirichlet boundary control problem. The EDG and IEDG methods are obtained from the HDG methods, and the global systems both use the same continuous elements; this reduces the number of degrees of freedom considerably. To approximate the control, we use continuous element in the  EDG method, and discontinuous elements in the IEDG method. Although the degrees of freedom of IEDG is slightly larger than the EDG method, the IEDG method  provides greater flexibility for  boundary control problems: we can use different finite element spaces for the control and the state. One possible benefit of the greater flexibility of the IEDG method is that discontinuous elements for the control may be better for more complicated problems (such as convection dominated problems) with sharp changes in the solution. For more details about the  EDG and IEDG methods; see \Cref{subsec:EmbeddedDGFormulations}.

Cockburn et al.\ in \cite{Cockburn_Guzman_Soon_EDG_SINUM_2009}  gave a rigorous error analysis of one EDG method for the Poisson equation. Recently, Zhang et al.\ in \cite{zhang2017optimal} proposed a new optimal EDG method for the Poisson equation. Moreover, we used these two EDG methods to approximate the solutions of  distributed control problems for the Poisson equation \cite{ZhengZhangSinglerEDG1} and a convection diffusion equation \cite{ZhengZhangSinglerEDG2}, respectively. However, the techniques in the previous EDG works are not applicable for the Dirichlet boundary control problem since the  regularity of the solution may be low. Instead of introducing a special projection as in \cite{HuMateosSinglerZhangZhang2}, we use an improved trace inequality from  \cite{Cai_He_Zhang_Adaptive_SINUM_2017} to deal with the low regularity solution. We improve the existing EDG error analysis by dealing with the case of low regularity solutions; also this is  the first work to give a rigorous error analysis for the IEDG method.  Moreover, in \Cref{sec:analysis} we prove the same convergence rates for the EDG and IEDG methods that we obtained for HDG methods in \cite{HuMateosSinglerZhangZhang1,HuMateosSinglerZhangZhang2}.  We present numerical results in \Cref{sec:numerics} for both diffusion dominated and convection dominated problems. Our experiments indicate that both methods work well for both cases; in addition, the IEDG method does a good job at computing sharp changes in the optimal control in the difficult convection dominated case.

\section{Background: Regularity and EDG Formulation }
\label{sec:Regularity and HDG Formulation}

Throughout,  the standard notation $H^{m}(\Omega)$ is used  for Sobolev spaces on $ \Omega $, and  we let $\|\cdot\|_{m,\Omega}$ and $|\cdot|_{m,\Omega}$ denote the Sobolev norm and seminorm.  We omit the index $m$ when $m = 0$ and the domain $\Omega$ if it will not cause confusion.  Also, set $H_0^1(\Omega)=\{v\in H^1(\Omega):v=0 \;\mbox{on}\; \partial \Omega\}$ and $ H(\text{div},\Omega) = \{\bm{v}\in [L^2(\Omega)]^d, \nabla\cdot \bm{v}\in L^2(\Omega)\} $.  We denote $(\cdot,\cdot)_{K}$ and $\langle\cdot,\cdot \rangle_{E}$ the standard  $L^2$-inner products on the domains $K\subset \mathbb R^d$ and $E\subset \mathbb R^{d-1}$.

Let $\omega  \  (1<\pi/\omega\leq 3)$ denote the largest interior angle of the domain $\Omega$, i.e., $\Omega$ is a \emph{convex} polygonal domain. Moreover, we assume $ \bm \beta $ satisfies
\begin{equation}\label{eqn:beta_assumptions1}
\bm \beta \in [L^\infty(\Omega)]^d,  \quad  \nabla\cdot\bm{\beta}\in L^\infty(\Omega),  \quad  \nabla \cdot \bm \beta \leq 0, \quad \nabla \nabla \cdot \bm{\beta}\in[L^2(\Omega)]^d. 
\end{equation}
%

%
The mixed weak form of the formal optimality system \eqref{eq_adeq_a}-\eqref{eq_adeq_e} is
\begin{subequations}\label{mixed}
	\begin{align}
	\varepsilon^{-1}\innOm{ {\bm q}}{\bm r} -\innOm{ y}{\nabla\cdot \bm r} + \innGa{ u}{\bm r\cdot \bm n} &= 0,\label{mixed_a}\\
	\innOm{\nabla\cdot( {\bm q}+\bm{\beta}  y)}{w} -\innOm{ y\nabla\cdot\bm{\beta}}{w}  &= \innOm{f}{w},\label{mixed_b}\\
	\varepsilon^{-1}\innOm{ {\bm p}}{\bm r}-\innOm{ z}{\nabla\cdot \bm r} &= 0,\label{mixed_c}\\%
	\innOm{\nabla\cdot( {\bm p}-\bm{\beta} z)}{w} &= \innOm{ y- {y_d}}{w},\label{mixed_d}\\
	\innGa{\gamma  u+ {\bm p}\cdot\bm n}{\mu} &=  0\label{mixed_e}
	\end{align}
	for all $(\bm r,w,\mu)\in H(\textup{div},\Omega)\times L^2(\Omega)\times L^2(\Gamma)$.
\end{subequations}
%

The following well-posedness and  regularity result is found in \cite{HuMateosSinglerZhangZhang1}.
\begin{theorem}\label{MT210}
	If $f=0$ and $y_d\in H^{t^*}(\Omega)$ for some $0\leq t^*<1$, then the optimal control problem \eqref{cost1}-\eqref{Ori_problem} has a unique solution $  u\in L^2(\Gamma)$ and $ u $ is uniquely determined by the optimality system \eqref{mixed_a}-\eqref{mixed_e}.  Moreover, for any $ s > 0 $ satisfying $s\leq  1/2 +t^*$ and $s<\min\{3/2,{\pi}/{\omega}-1/2\}$, we have $  u\in H^s (\Gamma)$ and 
	\begin{align*}
	(\bm q, {\bm p}, y,  z) &\in  [H^{s-\frac 12}(\Omega)]^d \cap H(\mathrm{div},\Omega)\times [H^{s+\frac 12}(\Omega)]^d\times H^{s+\frac 12}(\Omega)\times H^{s+\frac 32}(\Omega).
	\end{align*}
\end{theorem}
We note that the case of  $ f \neq 0 $ can be handled by the technique in \cite[pg.\ 3623]{Apel_Mateos_Pfefferer_Regularity_SICON_2015}.
\Cref{MT210} implies that if $ y_d \in H^{t^*}(\Omega) $ for some $ t^* \in (1/2,1) $, and $ \pi/3 <\omega < 2\pi/3 $, then $ u \in H^{r_u}(\Gamma) $ for some $ r_u \in (1, 3/2) $, we called this  the high regularity  case in \cite{HuMateosSinglerZhangZhang1}. In this scenario, $\bm{q} \in H^{r_{\bm q}}(\Omega)$ with $ r_{\bm q} > 1/2 $, which guarantees $ \bm q $ has a $L^2$ boundary trace.  We used this property to give a  convergence analysis of HDG methods in \cite{HuShenSinglerZhangZheng_HDG_Dirichlet_control1,HuMateosSinglerZhangZhang1}.

However, if $ t^* \in [0,1/2) $ or $2\pi/3 \leq \omega<\pi $, then we are in the low regularity case, i.e., $ u \in H^{r_u}(\Gamma) $ for some $ r_u \in [1/2, 1) $, and $ \|\bm q\|_{\partial\mathcal T_h} $ is not well-defined. The numerical analysis is more difficult in this case;  see \cite{HuMateosSinglerZhangZhang2} for an HDG method in the  low regularity case.

\subsection{A Class of Embedded DG Formulations}
\label{subsec:EmbeddedDGFormulations}
To better describe the  class of Embedded DG (EDG) methods, we first give some notation.

Let $\mathcal{T}_h$ be a conforming, quasi-uniform triangulation of $\Omega$. We denote by $\partial \mathcal{T}_h$ the set $\{\partial K: K\in \mathcal{T}_h\}$. For $K\in \mathcal T_h$, let $e = \partial K \cap \Gamma$ denote the boundary face of $ K $ if the $d-1$ Lebesgue measure of $e$ is non-zero. For two elements $K_1, K_2\in \mathcal T_h$, let $e = \partial K_1 \cap \partial K_2$ denote the interior face between $K_1$ and $K_2$ if the $d-1$ Lebesgue measure of $e$ is non-zero. Let $\mathcal E_h^o$ and $\mathcal E_h^{\partial}$ denote the sets of interior and boundary faces, respectively. We denote by $\mathcal E_h$ the union of  $\mathcal E_h^o$ and $\mathcal E_h^{\partial}$. Finally, we  introduce
\begin{align*}
(w,v)_{\mathcal{T}_h} = \sum_{K\in\mathcal{T}_h} (w,v)_K,   \quad\quad\quad\quad\left\langle \zeta,\rho\right\rangle_{\partial\mathcal{T}_h} = \sum_{K\in\mathcal{T}_h} \left\langle \zeta,\rho\right\rangle_{\partial K}.
\end{align*}

HDG methods were proposed by Cockburn et al.\ in \cite{Cockburn_Gopalakrishnan_Lazarov_Unify_SINUM_2009} as an improvement of traditional discontinuous Galerkin (DG) methods and have many applications; see, e.g., \cite{Qiu_Shen_Shi_elasticity_MathComp_2018,Cockburn_Gopalakrishnan_Nguyen_Peraire_Sayas_Stokes_MathComp_2011,Chabaud_Cockburn_Heat_MathComp_2012,Celiker_Cockburn_Shi_MathComp_2012, Cockburn_Johnny_Says_SINUM_2012, Cockburn_Shi_Stokes_MathComp_2013,Cockburn_Quenneville_Wave_MathComp_2014}. HDG methods are based on mixed formulations and introduce a new variable to approximate the trace of the scalar variable  along the element boundary.  To approximate the flux variable and solution, we use the discontinuous finite element spaces $\bm V_h$ and $W_h$:
\begin{align*}
\bm{V}_h  &:= \{\bm{v}\in [L^2(\Omega)]^d: \bm{v}|_{K}\in [\mathcal{P}^k(K)]^d, \forall K\in \mathcal{T}_h\},\\
{W}_h  &:= \{{w}\in L^2(\Omega): {w}|_{K}\in \mathcal{P}^{\ell}(K), \forall K\in \mathcal{T}_h\},
\end{align*}
where  $\mathcal{P}^k(K)$ denotes the set of polynomials of degree at most $k$ on a domain $K$.  HDG methods use the discontinuous finite element spaces to express  the approximate flux and solution in an element-by-element fashion in terms of numerical traces of the scalar variable. Then the globally coupled system only involves the numerical trace. The high number of globally coupled degrees of freedom is  significantly reduced compared to  other DG methods and standard mixed methods.

For the HDG methods, we use the following discontinuous finite element space to approximate the numerical trace:
\begin{align*}
M_h^{\textup{HDG}}:= \{{\mu}\in L^2(\mathcal E_h): {\mu}|_{e}\in \mathcal{P}^{m}(e), \forall e\in \mathcal E_h \}.
\end{align*}
Note that  $M_h^{\textup{HDG}}$ consists of functions which are discontinuous at the border of the faces.  Embedded discontinuous Galerkin (EDG) methods,  which were originally proposed in \cite{Guzey_Cockburn_Stolarski_EDG_IJNME_2007}, are obtained from HDG methods by replacing the discontinuous finite element space for the numerical traces with a continuous space, i.e.,
\begin{align*}
M_h^{\textup{EDG}}:= \{{\mu}\in\mathcal{C}^0 (\mathcal E_h): {\mu}|_{e}\in \mathcal{P}^{m}(e), \forall e\in \mathcal E_h \}.
\end{align*}
Hence, the number of degrees of freedom for the EDG method are much smaller than the HDG method, and also the same with the  CG method (after static condensation). The interior embedded discontinuous Galerkin (IEDG) method  was proposed and investigated for convection dominated flow problems in \cite{Nguyen_Peraire_Cockburn_JCP_2015}. The IEDG method is obtained by a simple change to  the space of the numerical trace from the HDG and EDG methods; specifically, 
\begin{align*}
M_h^{\textup{IEDG}}&:= \{{\mu}\in L^2(\mathcal E_h): {\mu}|_{e}\in \mathcal{P}^{m}(e), \forall e\in \mathcal E_h, \; \textup{and} \; \mu|_{\mathcal E_h^o} \in \mathcal{C}^0 (\mathcal E_h^o) \}.
\end{align*}
The functions in $M_h^{\textup{IEDG}}$ are only continuous on the union of interior edges and are  discontinuous on the union of the boundary edges. This simple change has many benefits even for pure PDE simulations; see \cite{Nguyen_Peraire_Cockburn_JCP_2015} for details. Compared to the EDG methods, the IEDG methods have a great potential for boundary control problems since they allow us to choose different spaces for the state and the control, as discussed in the introduction.

In this paper, we perform  a numerical analysis for both EDG and IEDG methods for the convection diffusion Dirichlet boundary control problem. To unify the analysis, we omit the superscripts  EDG  and IEDG on the space $M_h^{\textup{EDG}}$  and $M_h^{\textup{IEDG}}$, respectively.  We choose the following finite element spaces:
\begin{align*}
\bm{V}_h  &:= \{\bm{v}\in [L^2(\Omega)]^d: \bm{v}|_{K}\in [\mathcal{P}^k(K)]^d, \forall K\in \mathcal{T}_h\},\\
{W}_h  &:= \{{w}\in L^2(\Omega): {w}|_{K}\in \mathcal{P}^{k+1}(K), \forall K\in \mathcal{T}_h\},\\
M_h&:= \{{\mu}\in L^2(\mathcal E_h): {\mu}|_{e}\in \mathcal{P}^{k+1}(e), \forall e\in \mathcal E_h, \; \textup{and} \; \mu|_{\mathcal E_h^o} \in \mathcal{C}^0 (\mathcal E_h^o) \}.
\end{align*}

Let $M_h(o)$ and $M_h(\partial)$ be the spaces defined similarly to $M_h$ with $\mathcal E_h$ replaced by the set  of interior edges $\mathcal E_h^o$ and the set of boundary edges $\mathcal E_h^\partial$, respectively. The functions in $M_h(o)$ are continuous for both the EDG method and IEDG method,  while the functions in $M_h(\partial)$ are continuous across the boundary edges for the  EDG method and  discontinuous for the IEDG method. In addition, for any function $w\in W_h$ and $\bm r \in \bm V_h$, we use  $ \nabla w$  and $\nabla\cdot \bm r$ to denote the piecewise gradient and divergence on each element $K\in \mathcal T_h$, respectively.

Below, we consider the EDG and IEDG methods simultaneously; the choice of $M_h(\partial)$ determines the method as indicated above. The EDG  (or IEDG) method  seeks approximate fluxes ${\bm{q}}_h,{\bm{p}}_h \in \bm{V}_h $, states $ y_h, z_h \in W_h $, interior element boundary traces $ \widehat{y}_h^o,\widehat{z}_h^o \in {M}_h(o) $, and  control $ u_h \in {M}_h(\partial)$ satisfying
\begin{subequations}\label{HDG_discrete2}
	\begin{align}
	\varepsilon^{-1}(\bm q_h,\bm r_1)_{\mathcal T_h}-( y_h,\nabla\cdot\bm r_1)_{\mathcal T_h}+\langle \widehat y_h^o,\bm r_1\cdot\bm n\rangle_{\partial\mathcal T_h\backslash \mathcal E_h^\partial} + \langle  u_h,\bm r_1\cdot\bm n\rangle_{\mathcal E_h^\partial}&=0, \label{HDG_discrete2_a}\\
	-(\bm q_h+\bm \beta y_h,  \nabla w_1)_{\mathcal T_h} - (y_h \nabla\cdot\bm \beta, w_1)_{\mathcal T_h}+\langle\widehat {\bm q}_h\cdot\bm n,w_1\rangle_{\partial\mathcal T_h}  \quad  \nonumber \\ 
	+\langle \bm \beta\cdot\bm n\widehat y_h^o,w_1\rangle_{\partial\mathcal T_h\backslash\mathcal E_h^\partial} + \langle \bm \beta\cdot\bm n  u_h,w_1\rangle_{\mathcal E_h^\partial}  &=  ( f, w_1)_{\mathcal T_h},   \label{HDG_discrete2_b}
	%
	\end{align}
	for all $(\bm{r}_1, w_1)\in \bm{V}_h\times W_h$,
	\begin{align}
	\varepsilon^{-1}(\bm p_h,\bm r_2)_{\mathcal T_h}-(z_h,\nabla\cdot\bm r_2)_{\mathcal T_h}+\langle \widehat z_h^o,\bm r_2\cdot\bm n\rangle_{\partial\mathcal T_h\backslash\mathcal E_h^\partial}&=0,\label{HDG_discrete2_c}\\
	-(\bm p_h-\bm \beta z_h, \nabla w_2)_{\mathcal T_h}+\langle\widehat{\bm p}_h\cdot\bm n,w_2\rangle_{\partial\mathcal T_h}  \quad  \nonumber\\
	-\langle\bm \beta\cdot\bm n\widehat z_h^o,w_2\rangle_{\partial\mathcal T_h\backslash \mathcal E_h^\partial} - (y_h, w)_{\mathcal T_h}&=- (y_d, w_2)_{\mathcal T_h},  \label{HDG_discrete2_d}
	\end{align}
	for all $(\bm{r}_2, w_2)\in \bm{V}_h\times W_h$,
	\begin{align}
	\langle\widehat {\bm q}_h\cdot\bm n,\mu_1\rangle_{\partial\mathcal T_h\backslash\varepsilon^{\partial}_h}&=0\label{HDG_discrete2_e},\\
	\langle\widehat{\bm p}_h\cdot\bm n,\mu_2\rangle_{\partial\mathcal T_h\backslash\varepsilon^{\partial}_h}&=0,\label{HDG_discrete2_f}
	\end{align}
	for all $\mu_1,\mu_2\in {M}_h(o)$, and the optimality condition 
	\begin{align}
	\gamma \langle u_h, \mu_3 \rangle_{{\mathcal E_h^{\partial}}}+\langle \widehat{\bm{p}}_h\cdot \bm{n}, \mu_3\rangle_{{{\mathcal E_h^{\partial}}}} &=0, \label{HDG_discrete2_g}
	\end{align}
	for all $ \mu_3 \in M_h(\partial)$.   The numerical traces on $\partial\mathcal{T}_h$ are defined as 
	\begin{align}
	\widehat{\bm{q}}_h\cdot \bm n &=\bm q_h\cdot\bm n+(h^{-1} + \tau_1) (y_h-\widehat y_h^o)   \qquad \mbox{on} \; \partial \mathcal{T}_h\backslash\mathcal E_h^\partial, \label{EDG_discrete2_h}\\
	\widehat{\bm{q}}_h\cdot \bm n &=\bm q_h\cdot\bm n+(h^{-1}+\tau_1) (y_h- u_h) \qquad   \mbox{on}\;  \mathcal E_h^\partial, \label{HDG_discrete2_i}\\
	\widehat{\bm{p}}_h\cdot \bm n &=\bm p_h\cdot\bm n+(h^{-1}+\tau_2)(z_h-\widehat z_h^o) \qquad  \; \mbox{on} \; \partial \mathcal{T}_h\backslash\mathcal E_h^\partial,\label{EDG_discrete2_j}\\
	\widehat{\bm{p}}_h\cdot \bm n &=\bm p_h\cdot\bm n+ (h^{-1} + \tau_2)z_h \qquad\qquad \quad  \  \mbox{on}\;  \mathcal E_h^\partial,\label{EDG_discrete2_k}
	\end{align}
\end{subequations}
where $\tau_1$ and $\tau_2$ are positive stabilization functions defined on $\partial \mathcal T_h$ that  satisfy
\begin{align}\label{tau1tau2}
\tau_2 = \tau_1 - \bm \beta \cdot \bm n. 
\end{align}
The condition \eqref{tau1tau2} for the stabilization functions $\tau_1$ and $\tau_2$ has now been used in a number of works; see, e.g., \cite{HuMateosSinglerZhangZhang1,HuMateosSinglerZhangZhang2} for convection diffusion Dirichlet boundary control problems and \cite{ChenHuShenSinglerZhangZheng_HDG_Convection_Dirtributed_Control_JCAM_2018,HuShenSinglerZhangZheng_HDG_Dirichlet_control3,ZhengZhangSinglerEDG2} for convection diffusion distributed optimal control problems. This condition causes the optimize-then-discretize and discretize-then-optimize EDG/HDG approaches to the control problem to produce equivalent results; see \cite{ZhengZhangSinglerEDG2} for details concerning an EDG method for a distributed convection diffusion optimal control problem. Our implementation of the EDG and IEDG methods is similar to the HDG implementation for a Poisson Dirichlet boundary control problem described in our earlier work \cite{HuShenSinglerZhangZheng_HDG_Dirichlet_control1}.

\section{Error Analysis}
\label{sec:analysis}
Next, we provide a convergence analysis of the above EDG and IEDG methods for the convection diffusion Dirichlet boundary control problem in both high regularity and low regularity cases. For the high regularity case, tools from the  analysis technique in \cite{ZhengZhangSinglerEDG2} for a convection diffusion distributed control problem can be modified to apply to the Dirichlet boundary control problem. For the low regularity case, we introduced a speical projection operator in our earlier HDG work \cite{HuMateosSinglerZhangZhang2} to avoid the quantity $\|\bm q\cdot\bm n\|_{\partial \mathcal T_h}$ in the analysis; however, this complicated the analysis. In this work, we use an improved inverse inequality from \cite{Cai_He_Zhang_Adaptive_SINUM_2017}, and simplify the error analysis for the low regularity case.  It is worth mentioning that part of our analysis (step 1 to step 3 in \Cref{ProofofMainResult}) improves the existing EDG error analysis by dealing with the case of low regularity solutions. In this work, we only perform an error analysis for the diffusion dominated case; i.e., in this section, we assume $\varepsilon=\mathcal O(1)$. The generic constant $C$  may depend on the data of the problem but is independent of $h$ and may change from line to line.

\subsection{Assumptions and Main Result}
We assume the solution of the optimality system \eqref{mixed_a}-\eqref{mixed_e} has the following regularity properties:
\begin{subequations}\label{eqn:regularity2}
	\begin{gather}
	\bm q \in [ H^{r_{\bm q}}(\Omega) ]^d \cap H(\mathrm{div},\Omega),  \  \bm p \in [H^{r_{\bm p}}(\Omega)]^d, \ y \in H^{r_y}(\Omega),   \    z \in H^{r_z}(\Omega),  \\
	r_{\bm q} >  0,  \quad  r_{\bm p} > 1, \quad r_y > 1,  \quad  r_z > 2.\label{eqn:s_rates_ineq}
	\end{gather}
\end{subequations}
In the 2D case, \Cref{MT210} guarantees this regularity condition is satisfied.

We now state our main convergence result.
\begin{theorem}\label{main_res}
	Let
	\begin{equation}\label{eqn:s_rates}
	\begin{split}
	s_{\bm q} &= \min\{r_{\bm q}, 1\}, \qquad  s_{y} = \min\{r_{y}, k+2\}, \\
	s_{\bm p} &= \min\{r_{\bm p}, k+1\},  \qquad s_{z} = \min\{r_{z}, k+2\}.
	\end{split}
	\end{equation}
	we have 
	\begin{align*}
	\norm{u-u_h}_{\mathcal E_h^\partial}&\le C (h^{s_{\bm q}+\frac 1 2}\norm{\bm q}_{s_{\bm q}} + h^{s_{y}-\frac 12 }\norm{y}_{s_{y}} + h^{s_{\bm p}-\frac 1 2}\norm{\bm p}_{s_{\bm p}} +  h^{s_{z}-\frac 3 2}\norm{z}_{s_{z}}),\\
	\norm {y-y_h}_{\mathcal T_h} &\le C (h^{s_{\bm q}+\frac 1 2}\norm{\bm q}_{s_{\bm q}} + h^{s_{y}-\frac 12 }\norm{y}_{s_{y}} + h^{s_{\bm p}-\frac 1 2}\norm{\bm p}_{s_{\bm p}} +  h^{s_{z}-\frac 3 2}\norm{z}_{s_{z}}),\\
	\norm {\bm p - \bm p_h}_{\mathcal T_h}  &\le C (h^{s_{\bm q}+\frac 1 2}\norm{\bm q}_{s_{\bm q}} + h^{s_{y}-\frac 12 }\norm{y}_{s_{y}} + h^{s_{\bm p}-\frac 1 2}\norm{\bm p}_{s_{\bm p}} +  h^{s_{z}-\frac 3 2}\norm{z}_{s_{z}}),\\
	\norm {z - z_h}_{\mathcal T_h} & \le C  (h^{s_{\bm q}+\frac 1 2}\norm{\bm q}_{s_{\bm q}} + h^{s_{y}-\frac 12 }\norm{y}_{s_{y}} + h^{s_{\bm p}-\frac 1 2}\norm{\bm p}_{s_{\bm p}} +  h^{s_{z}-\frac 3 2}\norm{z}_{s_{z}}).
	\end{align*}
	If $ k \geq 1 $, then
	\begin{align*}
	\norm {\bm q - \bm q_h}_{\mathcal T_h} \le C (h^{s_{\bm q}}\norm{\bm q}_{s_{\bm q}} + h^{s_{y}-1 }\norm{y}_{s_{y}} + h^{s_{\bm p}-1}\norm{\bm p}_{s_{\bm p}} +  h^{s_{z}-2}\norm{z}_{s_{z}}).
	\end{align*}
\end{theorem}

Specializing to the 2D case gives the following result:
\begin{corollary}\label{cor:main_result}
	Suppose $ d = 2 $, $ f = 0 $ and $ y_d \in H^{t^*}(\Omega) $ for some $ t^* \in (0,1) $.  Let $ \pi/3\le \omega<\pi $ be the largest interior angle of $\Gamma$, and let $ r > 0 $ satisfy
	$$
	r \leq  r_d := \frac{1}{2} + t^* \in (1/2,3/2),  \quad  \mbox{and}  \quad  r < r_{\Omega} := \min\left\{ \frac{3}{2}, \frac{\pi}{\omega} - \frac{1}{2} \right\} \in (1/2, 3/2].
	$$
	If $ k = 1 $, then
	\begin{align*}
	\norm{u-u_h}_{\mathcal E_h^\partial}&\le C  h^{r} (\norm{\bm q}_{H^{r-1/2}} + \norm{y}_{H^{r+1/2}} + \norm{\bm p}_{H^{r+1/2}} +  \norm{z}_{H^{r+3/2}}),\\
	\norm{y-y_h}_{\mathcal T_h}&\le C  h^{r} (\norm{\bm q}_{H^{r-1/2}} + \norm{y}_{H^{r+1/2}} + \norm{\bm p}_{H^{r+1/2}} +  \norm{z}_{H^{r+3/2}}),\\
	\norm {\bm p - \bm p_h}_{\mathcal T_h}   &\le C  h^{r} (\norm{\bm q}_{H^{r-1/2}} + \norm{y}_{H^{r+1/2}} + \norm{\bm p}_{H^{r+1/2}} +  \norm{z}_{H^{r+3/2}}),\\
	\norm {z - z_h}_{\mathcal T_h}   & \le C  h^{r} (\norm{\bm q}_{H^{r-1/2}} + \norm{y}_{H^{r+1/2}} + \norm{\bm p}_{H^{r+1/2}} +  \norm{z}_{H^{r+3/2}}).
	\end{align*}
	If in addition $ r > 1/2 $, then
	$$
	\norm {\bm q - \bm q_h}_{\mathcal T_h}  \le C  h^{r-1/2} (\norm{\bm q}_{H^{r-1/2}} + \norm{y}_{H^{r+1/2}} + \norm{\bm p}_{H^{r+1/2}} +  \norm{z}_{H^{r+3/2}}).
	$$
	Furthermore, if $ k = 0 $ then
	\begin{align*}
	\norm{u-u_h}_{\mathcal E_h^\partial}&\le C h^{1/2} ( \norm{\bm q}_{H^{r-1/2}} + \norm{y}_{H^{r+1/2}} + \norm{\bm p}_{H^{1}} +  \norm{z}_{H^{2}}),\\
	\norm{y-y_h}_{\mathcal T_h}&\le C h^{1/2} ( \norm{\bm q}_{H^{r-1/2}} + \norm{y}_{H^{r+1/2}} + \norm{\bm p}_{H^{1}} +  \norm{z}_{H^{2}}),\\
	\norm {\bm p - \bm p_h}_{\mathcal T_h}   &\le C h^{1/2} ( \norm{\bm q}_{H^{r-1/2}} + \norm{y}_{H^{r+1/2}} + \norm{\bm p}_{H^{1}} +  \norm{z}_{H^{2}}),\\
	\norm {z - z_h}_{\mathcal T_h}   & \le C h^{1/2} ( \norm{\bm q}_{H^{r-1/2}} + \norm{y}_{H^{r+1/2}} + \norm{\bm p}_{H^{1}} +  \norm{z}_{H^{2}}).
	\end{align*}
\end{corollary}
\begin{remark}
	As in \cite{HuMateosSinglerZhangZhang1,HuMateosSinglerZhangZhang2}, when $ k = 1 $ the convergence rates are optimal for the control and the flux $ \bm q $ and suboptimal for the other variables. Compared to the HDG method used in \cite{HuMateosSinglerZhangZhang1,HuMateosSinglerZhangZhang2}, we obtain the same convergence rates for the  EDG and IEDG methods.
\end{remark}

\subsection{Preliminary Material}
\label{sec:Projectionoperator}
We introduce the standard $L^2$-orthogonal projection operators $\bm{\Pi}_h^k: [L^2(K)]^d\to [\mathcal P^k(K)]^d$ and $\Pi_h^{k+1}: L^2(K)\to \mathcal P^{k+1}(K)$, which satisfy
\begin{subequations} \label{def_L2}
	\begin{align}
	(\bm \Pi_h^k \bm q, \bm r)_{K}&=(\bm q,\bm r)_{K} \quad \forall \bm r\in [{\mathcal{ P}}^k(K)]^d,\\
	(\Pi_h^{k+1} y,w)_{K}&=(y,w)_{K}\quad \forall w\in \mathcal{P}^{k+1}(K).
	\end{align}
\end{subequations}
Moreover, we use the following well-known bounds:
\begin{subequations}\label{classical_ine}
	\begin{align}
	\norm {\bm q -\bm\Pi_h^k \bm q}_{\mathcal T_h} &\le Ch^{s_{\bm q}} \norm{\bm q}_{s_{\bm q}},\quad\norm {y -{\Pi_h^{k+1} y}}_{\mathcal T_h} \le Ch^{s_{y}} \norm{y}_{s_{y}},\\
	\norm {y -{\Pi_h^{k+1} y}}_{\partial\mathcal T_h} &\le Ch^{s_{y}-\frac 1 2} \norm{y}_{s_{y}}, \  
	\norm {w}_{\partial \mathcal T_h} \le Ch^{-\frac 12} \norm {w}_{ \mathcal T_h}, \  \forall w\in W_h.
	\end{align}
\end{subequations}
where $ s_{\bm q} $ and $ s_y $ are defined in \Cref{main_res}.  We have the same projection error bounds for $\bm p$ and $z$. 

Since we only assume $y\in H^{r_y}(\Omega)$ with $r_y>1$, certain components of the solution may not be continuous; for example, we cannot guarantee $y$ is continuous on $\Omega$ when $d=3$. Therefore, the standard Lagrange interpolation operator is not applicable; hence we utilize the Scott-Zhang interpolation operator $\mathcal I_h^{k+1}: H^{1}(\Omega)\to \widetilde W_h$ from \cite{Scott_Zhang_Ih_MathComp_1990}, where 
\begin{align*}
\widetilde W_h:=\{w\in \mathcal C^0(\Omega): w|_K\in \mathcal P^{k+1}(K)\}.
\end{align*}
The following bound is found in \cite[Theorem 4.1]{Scott_Zhang_Ih_MathComp_1990}:
\begin{align}\label{scottzhang2}
\norm {y -{ \mathcal I_h^{k+1} y}}_{\mathcal T_h} \le C  h^{s_y} \norm{y}_{s_y}.
\end{align}
By an inverse inequality,  a trace inequality and \Cref{scottzhang2} we obtain
\begin{align}\label{scottzhang}
\norm {y -{ \mathcal I_h^{k+1} y}}_{\partial \mathcal T_h} \le C  h^{s_y-1/2} \norm{y}_{s_y}.
\end{align}

Next, for any $(\bm q_h,y_h,\widehat y_h^o;\bm r_1,w_1,\mu_1)\in [\bm V_h\times W_h\times M_h(o)]^2$ and $(\bm p_h,z_h,\widehat z_h^o;\bm r_2,\\ w_2,\mu_2)\in [\bm V_h\times W_h\times M_h(o)]^2$ ,  define  the operators  $ \mathscr B_1$ and $ \mathscr B_2 $ by 
\begin{align}
\hspace{1em}&\hspace{-1em} \mathscr  B_1( \bm q_h,y_h,\widehat y_h^o;\bm r_1,w_1,\mu_1) \nonumber\\
& =\varepsilon^{-1} (\bm q_h,\bm r_1)_{\mathcal T_h}-( y_h,\nabla\cdot\bm r_1)_{\mathcal T_h}+\langle \widehat y_h^o,\bm r_1\cdot\bm n\rangle_{\partial\mathcal T_h\backslash \mathcal E_h^\partial} \nonumber\\ 
& \quad  +(\nabla\cdot \bm q_h,w_1)_{\mathcal{T}_h} -(\bm \beta y_h,  \nabla w_1)_{\mathcal T_h}-(\nabla\cdot\bm\beta y_h,w_1)_{\mathcal T_h}\nonumber\\
& \quad +\langle (h^{-1} + \tau_1)y_h, w_1\rangle_{\partial\mathcal T_h}  -\langle (h^{-1}+\tau_1 - \bm{\beta}\cdot \bm n) \widehat y_h^o,w_1\rangle_{\partial\mathcal T_h\backslash \mathcal E_h^\partial}\nonumber\\
& \quad -\langle \bm q_h\cdot \bm n +(h^{-1} + \tau_1)(y_h-\widehat y_h^o),\mu_1\rangle_{\partial\mathcal T_h\backslash\mathcal E_h^{\partial}},\label{def_B1}\\
\hspace{1em}&\hspace{-1em} \mathscr B_2 (\bm p_h,z_h,\widehat z_h^o;\bm r_2, w_2,\mu_2)\nonumber\\
&=\varepsilon^{-1} (\bm p_h,\bm r_2)_{\mathcal T_h}-( z_h,\nabla\cdot\bm r_2)_{\mathcal T_h}+\langle \widehat z_h^o,\bm r_2\cdot\bm n\rangle_{\partial\mathcal T_h\backslash\mathcal E_h^\partial}\nonumber\\
&\quad +(\nabla\cdot\bm p_h,w_2)_{\mathcal T_h}  +(\bm \beta z_h,  \nabla w_2)_{\mathcal T_h}+\langle (h^{-1}+\tau_2) z_h, w_2\rangle_{\partial\mathcal T_h}\nonumber\\
&\quad -\langle ( h^{-1}+\tau_2 + \bm \beta\cdot \bm n)\widehat z_h^o ,w_2\rangle_{\partial\mathcal T_h\backslash\mathcal E_h^\partial}\nonumber\\
&  \quad -\langle  {\bm p}_h\cdot\bm n +(h^{-1} + \tau_2)(z_h-\widehat z_h^o), \mu_2\rangle_{\partial\mathcal T_h\backslash\mathcal E_h^{\partial}}\label{def_B2}.
\end{align}

Using this definition, we rewrite the  EDG (or IEDG) optimality system \eqref{HDG_discrete2} as follows: find $({\bm{q}}_h,{\bm{p}}_h,y_h,z_h,\widehat y_h^o,\widehat z_h^o,u_h)\in \bm{V}_h\times\bm{V}_h\times W_h \times W_h\times {M}_h(o)\times {M}_h(o)\times {M}_h(\partial)$  such that
\begin{subequations}\label{IEDG_full_discrete}
	\begin{align}
	\mathscr B_1 (\bm q_h,y_h,\widehat y_h^o;\bm r_1,w_1,\mu_1) &= -\langle u_h, \bm r_1\cdot \bm{n} - (h^{-1}+\tau_1 - \bm{\beta} \cdot\bm n) w_1 \rangle_{{\mathcal E_h^{\partial}}}\nonumber\\
	&\quad +(f, w_1)_{{\mathcal{T}_h}}, \label{IEDG_full_discrete_a}\\
	\mathscr B_2(\bm p_h,z_h,\widehat z_h^o;\bm r_2,w_2,\mu_2) &= (y_h -y_d,w_2)_{\mathcal T_h},\label{IEDG_full_discrete_b}\\
	\langle {\bm{p}}_h\cdot \bm{n} + (h^{-1}+\tau_2)  z_h, \mu_3\rangle_{{{\mathcal E_h^{\partial}}}} &= -\gamma \langle u_h, \mu_3 \rangle_{{\mathcal E_h^{\partial}}}, \label{IEDG_full_discrete_c}
	\end{align}
\end{subequations}
for all $\left(\bm{r}_1, \bm{r}_2, w_1, w_2, \mu_1, \mu_2, {\mu}_3\right)\in \bm{V}_h\times\bm{V}_h\times W_h \times W_h\times {M}_h(o)\times {M}_h(o)\times {M}_h(\partial)$.

Next, we present three basic but fundamental results. The proofs follow similar arguments in \cite{HuShenSinglerZhangZheng_HDG_Dirichlet_control1,HuMateosSinglerZhangZhang1,HuMateosSinglerZhangZhang2} and are omitted. 
\begin{lemma}\label{property_B}
	For any $ ( \bm v_h, w_h, \mu_h ) \in \bm V_h \times W_h \times {M}_h(o) $, we have
	\begin{align*}
	\hspace{1em}&\hspace{-1em} \mathscr B_1(\bm v_h,w_h,\mu_h;\bm v_h,w_h,\mu_h)\\
	&=\varepsilon^{-1}(\bm v_h,\bm v_h)_{\mathcal T_h}+ \langle (h^{-1}+\tau_1 - \frac 12 \bm \beta\cdot\bm n)(w_h-\mu_h),w_h-\mu_h\rangle_{\partial\mathcal T_h\backslash \mathcal E_h^\partial}\\
	&\quad-\frac 1 2(\nabla\cdot\bm\beta w_h,w_h)_{\mathcal T_h} +\langle (h^{-1}+\tau_1-\frac12\bm \beta\cdot\bm n) w_h,w_h\rangle_{\mathcal E_h^\partial},\\
	\hspace{1em}&\hspace{-1em}\mathscr B_2(\bm v_h,w_h,\mu_h;\bm v_h,w_h,\mu_h)\\
	&=\varepsilon^{-1}(\bm v_h,\bm v_h)_{\mathcal T_h}+ \langle (h^{-1}+\tau_2 + \frac 12 \bm \beta\cdot\bm n)(w_h-\mu_h),w_h-\mu_h\rangle_{\partial\mathcal T_h\backslash \mathcal E_h^\partial}\\
	&\quad-\frac 1 2(\nabla\cdot\bm\beta w_h,w_h)_{\mathcal T_h} +\langle (h^{-1}+\tau_2+\frac12\bm \beta\cdot\bm n) w_h,w_h\rangle_{\mathcal E_h^\partial}.
	\end{align*}
\end{lemma}
%
\begin{lemma}\label{identical_equa}
	For any $ ( \bm v_1, \bm v_2,w_1, w_2,  \mu_1,\mu_2 ) \in \bm V_h\times \bm  V_h \times W_h\times W_h \times {M}_h(o)\times {M}_h(o) $, we have
	\begin{align*}
	\mathscr B_1 (\bm v_1,w_1,\mu_1;\bm v_2,-w_2,-\mu_2) + \mathscr B_2 (\bm v_2,w_2,\mu_2;-\bm v_1,w_1,\mu_1) = 0.
	\end{align*}
\end{lemma}

\begin{proposition}\label{ex_uni}
	There exists a unique solution of the discrete system \eqref{IEDG_full_discrete}.
\end{proposition}

Next, we introduce the improved trace inequality.
\begin{lemma}\cite[Lemma 2.4]{Cai_He_Zhang_Adaptive_SINUM_2017}\label{improved_trace_inequality}
	Let $E$ be a face of $K\in\mathcal T_h$. If $\bm q \in  [H^{s_{\bm q}}(\Omega)]^d\cap H(\textup{div}, \Omega)$ with $s_{\bm q}>0$, then for all $\mu\in \mathcal P^{k+1}(E)$, we have 
	\begin{align}
	\langle \bm q\cdot\bm n, \mu \rangle_E \le C h^{-1/2} \|\mu\|_{E}(\|\bm q\|_{K}+h\|\nabla\cdot\bm q\|_K).
	\end{align}
	
\end{lemma}

\subsection{Proof of \Cref{main_res}}
\label{ProofofMainResult}

To prove \Cref{main_res}, we follow the strategy in \cite{HuShenSinglerZhangZheng_HDG_Dirichlet_control1} and split the proof into
seven steps.  We consider the following auxiliary problem: find $$({\bm{q}}_h(u),{\bm{p}}_h(u), y_h(u), z_h(u), {\widehat{y}}_h^o(u), {\widehat{z}}_h^o(u))\in \bm{V}_h\times\bm{V}_h\times W_h \times W_h\times {M}_h(o)\times {M}_h(o)$$ such that
\begin{subequations}\label{HDG_inter_u}
	\begin{align}
	\mathscr B_1(\bm q_h(u),y_h(u),\widehat{y}_h^o(u);\bm r_1, w_1,\mu_1)&= - \langle  u, \bm r_1\cdot\bm n - (h^{-1} + \tau_1 - \bm{\beta}\cdot\bm n) w_1 \rangle_{\mathcal E_h^\partial}\nonumber\\
	&\quad  + ( f ,w_1)_{\mathcal T_h} ,\label{IEDG_u_a} \\
	\mathscr B_2(\bm p_h(u),z_h(u),\widehat{z}_h^o(u);\bm r_2, w_2,\mu_2)&=(y_h(u) - y_d, w_2)_{\mathcal T_h},\label{IEDG_u_b}
	\end{align}
\end{subequations}
for all $\left(\bm{r}_1, \bm{r}_2,w_1,w_2,\mu_1,\mu_2\right)\in \bm{V}_h\times\bm{V}_h \times W_h\times W_h\times {M}_h(o)\times {M}_h(o)$.  We begin by bounding the error between the solutions of the auxiliary problem \eqref{HDG_inter_u} and the mixed form \eqref{mixed_a}-\eqref{mixed_d} of the optimality system.

\subsubsection{Step 1: The error equation for part 1 of the auxiliary problem \eqref{IEDG_u_a}.} \label{subsec:proof_step1}

\begin{lemma} \label{Pro_B1}
	For all $\left(\bm{r}_1,w_1,\mu_1\right)\in \bm{V}_h \times W_h\times {M}_h(o)$, we have
	\begin{align*}
	\hspace{1em}&\hspace{-1em}  \mathscr B_1 (\bm \Pi_h^0 {\bm q},\Pi_h^{k+1} { y}, \mathcal I_h^{k+1}  y, \bm r_1, w_1, \mu_1)\\
	&=\langle u,(h^{-1} +\tau_1-\bm \beta \cdot \bm n ) w_1-\bm r_1\cdot \bm n  \rangle_{\mathcal E_h^{\partial}} + (f,w_1)_{\mathcal T_h}-\varepsilon^{-1}( {\bm q} - \bm \Pi_h^0 \bm q, \bm{r}_1)_{{\mathcal{T}_h}}\\
	&\quad  +\langle  \mathcal I_h^{k+1} y - y, \bm r_1\cdot \bm{n} \rangle_{\partial{{\mathcal{T}_h}}\backslash {\mathcal E_h^{\partial}}} - (\bm \Pi_h^0 \bm q - \bm q,\nabla w_1)_{\mathcal{T}_h}\\
	&\quad +  (  \bm \beta (y - \Pi_h^{k+1} y),  \nabla w_1)_{{\mathcal{T}_h}}  + (\nabla\cdot \bm \beta ( y - \Pi_h^{k+1} y), w_1)_{\mathcal T_h}\\
	&\quad +\langle (h^{-1}+\tau_1)(\Pi_h^{k+1} y -\mathcal I_h^{k+1}  y), w_1 - \mu_1 \rangle_{\partial{{\mathcal{T}_h}}\backslash\mathcal E_h^\partial}   \\
	&\quad +\langle (h^{-1}+\tau_1)(\Pi_h^{k+1} y - y), w_1 \rangle_{\mathcal E_h^\partial}\\
	&\quad  + \langle \bm{\beta}\cdot\bm n (\mathcal I_h^{k+1}  y-y), w_1-\mu_1 \rangle_{\partial{{\mathcal{T}_h}}\backslash \mathcal E_h^{\partial}} +\langle  (\bm \Pi_h^0 \bm q - \bm q)\cdot\bm n, w_1 - \mu_1\rangle_{\partial\mathcal T_h\backslash\mathcal E_h^\partial}\\
	&\quad  +\langle  (\bm \Pi_h^0 \bm q - \bm q)\cdot\bm n, w_1\rangle_{\mathcal  E_h^\partial}.
	\end{align*}
\end{lemma}
\begin{proof}
	Using the definition of $ \mathscr B_1 $ in \eqref{def_B1} gives
	\begin{align*}
	\hspace{1em}&\hspace{-1em}  \mathscr B_1 (\bm \Pi_h^0 {\bm q},\Pi_h^{k+1} { y}, \mathcal I_h^{k+1} y, \bm r_1, w_1, \mu_1)\\
	&= \varepsilon^{-1}(\bm \Pi_h^0 {\bm q}, \bm{r}_1)_{{\mathcal{T}_h}}- (\Pi_h^{k+1} { y}, \nabla\cdot \bm{r}_1)_{{\mathcal{T}_h}}+\langle \mathcal I_h^{k+1}  y, \bm{r}_1\cdot \bm{n} \rangle_{\partial{{\mathcal{T}_h}}\backslash {\mathcal E_h^{\partial}}}\\
	&  \quad+(\nabla\cdot\bm \Pi_h^0 {\bm q},w_1)_{\mathcal{T}_h}- ( \bm{\beta} \Pi_h^{k+1} y, \nabla w_1)_{{\mathcal{T}_h}}- (\nabla\cdot\bm{\beta} \Pi_h^{k+1} y,  w_1)_{{\mathcal{T}_h}}\\
	&\quad+\langle (h^{-1} + \tau_1) \Pi_h^{k+1} { y}, w_1 \rangle_{\partial{{\mathcal{T}_h}}}
	- (h^{-1}+\tau_1-\bm \beta\cdot\bm n) \mathcal I_h^{k+1}  y, w_1 \rangle_{\partial{{\mathcal{T}_h}}\backslash \mathcal E_h^{\partial}}\\
	&\quad	-\langle  \bm \Pi_h^0 \bm q\cdot\bm n+(h^{-1}+\tau_1)(\Pi_h^{k+1} y - \mathcal I_h^{k+1} y),\mu_1\rangle_{\partial\mathcal T_h\backslash\mathcal E_h^{\partial}}.
	\end{align*}
	Using properties of the $ L^2 $ projections \eqref{def_L2} gives
	\begin{align*}
	\hspace{1em}&\hspace{-1em} \mathscr B_1 (\bm \Pi_h^0 {\bm q},\Pi_h^{k+1} { y}, \mathcal I_h^{k+1} y, \bm r_1, w_1, \mu_1) \\
	&= \varepsilon^{-1}( {\bm q}, \bm{r}_1)_{{\mathcal{T}_h}}- ({ y}, \nabla\cdot \bm r_1)_{{\mathcal{T}_h}}+\langle   y, \bm r_1\cdot \bm{n} \rangle_{\partial{{\mathcal{T}_h}}\backslash {\mathcal E_h^{\partial}}} \\
	&\quad -\varepsilon^{-1}( {\bm q} - \bm \Pi_h^0 \bm q, \bm{r}_1)_{{\mathcal{T}_h}} +\langle  \mathcal I_h^{k+1} y - y, \bm r_1\cdot \bm{n} \rangle_{\partial{{\mathcal{T}_h}}\backslash {\mathcal E_h^{\partial}}}\\
	& \quad   +(\nabla\cdot\bm q,w_1)_{\mathcal{T}_h} + (\nabla\cdot (\bm \Pi_h^0 \bm q - \bm q),w_1)_{\mathcal{T}_h}- ( \bm \beta y, \nabla w_1)_{{\mathcal{T}_h}} \\
	&\quad +  (  \bm \beta (y - \Pi_h^{k+1} y),  \nabla w_1)_{{\mathcal{T}_h}} - (\nabla\cdot\bm{\beta} y, w_1)_{\mathcal T_h} + (\nabla\cdot \bm \beta ( y - \Pi_h^{k+1} y), w_1)_{\mathcal T_h} \\
	&\quad +\langle (h^{-1}+\tau_1)(\Pi_h^{k+1} y - y), w_1 \rangle_{\partial{{\mathcal{T}_h}}} - \langle (h^{-1}+\tau_1) (\mathcal I_h^{k+1}  y - y), w_1 \rangle_{\partial{{\mathcal{T}_h}}\backslash \mathcal E_h^{\partial}} \\
	& \quad  +  \langle (h^{-1}+\tau_1) y, w_1 \rangle_{\mathcal E_h^{\partial}} 
	+\langle \bm{\beta}\cdot\bm n y, w_1 \rangle_{\partial\mathcal T_h\backslash\mathcal E_h^\partial} + \langle \bm{\beta}\cdot\bm n (\mathcal I_h^{k+1} y - y), w_1 \rangle_{\partial\mathcal T_h\backslash\mathcal E_h^\partial}\\
	&\quad -\langle  \bm \Pi^0 \bm q\cdot\bm n+(h^{-1}+\tau_1)(\Pi_h^{k+1} y - \mathcal I_h^{k+1} y),\mu_1\rangle_{\partial\mathcal T_h\backslash\mathcal E_h^{\partial}}.
	\end{align*}
	The flux $\bm{q}$ and  state $ y $ satisfy
	\begin{align*}
	\varepsilon^{-1}(\bm{q},\bm{r}_1)_{\mathcal{T}_h}-(y,\nabla\cdot \bm{r}_1)_{\mathcal{T}_h}+\left\langle{y},\bm r_1\cdot \bm n \right\rangle_{\partial {\mathcal{T}_h}\backslash\mathcal E_h^\partial} &= -\langle u,\bm r_1\cdot \bm n \rangle_{\mathcal E_h^\partial},\\
	(\nabla\cdot \bm{q},w_1)_{\mathcal{T}_h}-(\bm{\beta} y,\nabla w_1)_{\mathcal{T}_h}-(\nabla \cdot \bm{\beta} y, w_1)_{\mathcal{T}_h}\\
	+\left\langle \bm \beta\cdot \bm n y,w_1\right\rangle_{\partial {\mathcal{T}_h}\backslash \mathcal E_h^\partial} &= -\langle \bm \beta \cdot \bm n u,w_1 \rangle_{\mathcal E_h^\partial}+(f,w_1)_{\mathcal{T}_h},\\
	\langle \bm q \cdot \bm n, \mu_1\rangle_{\partial {\mathcal{T}_h}\backslash \mathcal E_h^\partial} &= 0,
	\end{align*}
	for all $(\bm{r}_1,w_1,\mu_1)\in\bm{V}_h\times W_h\times M_h(o)$. This gives
	\begin{align*}
	\hspace{1em}&\hspace{-1em}  \mathscr B_1 (\bm \Pi_h^0 {\bm q},\Pi_h^{k+1} { y}, \mathcal I_h^{k+1}  y, \bm r_1, w_1, \mu_1)\\
	&=\langle u,(h^{-1} +\tau_1-\bm \beta \cdot \bm n ) w_1-\bm r_1\cdot \bm n  \rangle_{\mathcal E_h^{\partial}} + (f,w_1)_{\mathcal T_h} -\varepsilon^{-1}( {\bm q} - \bm \Pi_h^0 \bm q, \bm{r}_1)_{{\mathcal{T}_h}}\\
	&\quad  +\langle  \mathcal I_h^{k+1} y - y, \bm r_1\cdot \bm{n} \rangle_{\partial{{\mathcal{T}_h}}\backslash {\mathcal E_h^{\partial}}} + (\nabla\cdot (\bm \Pi_h^0 \bm q - \bm q),w_1)_{\mathcal{T}_h} \\
	&\quad +  (\bm \beta (y - \Pi_h^{k+1} y),  \nabla w_1)_{{\mathcal{T}_h}}  + (\nabla\cdot \bm \beta ( y - \Pi_h^{k+1} y), w_1)_{\mathcal T_h}\\
	&\quad +\langle (h^{-1}+\tau_1)(\Pi_h^{k+1} y -\mathcal I_h^{k+1}  y), w_1 - \mu_1 \rangle_{\partial{{\mathcal{T}_h}}\backslash\mathcal E_h^\partial}  \\
	&\quad +\langle (h^{-1}+\tau_1)(\Pi_h^{k+1} y - y), w_1 \rangle_{\mathcal E_h^\partial} + \langle \bm{\beta}\cdot\bm n (\mathcal I_h^{k+1}  y-y), w_1 \rangle_{\partial{{\mathcal{T}_h}}\backslash \mathcal E_h^{\partial}}\\
	&\quad -\langle ( \bm \Pi_h^0 \bm q - \bm q)\cdot\bm n,\mu_1\rangle_{\partial\mathcal T_h\backslash\mathcal E_h^{\partial}}\\
	&=\langle u,(h^{-1} +\tau_1-\bm \beta \cdot \bm n ) w_1-\bm r_1\cdot \bm n  \rangle_{\mathcal E_h^{\partial}} + (f,w_1)_{\mathcal T_h}-\varepsilon^{-1}( {\bm q} - \bm \Pi_h^0 \bm q, \bm{r}_1)_{{\mathcal{T}_h}}\\
	&\quad  +\langle  \mathcal I_h^{k+1} y - y, \bm r_1\cdot \bm{n} \rangle_{\partial{{\mathcal{T}_h}}\backslash {\mathcal E_h^{\partial}}} - (\bm \Pi_h^0 \bm q - \bm q,\nabla w_1)_{\mathcal{T}_h}\\
	&\quad +  (  \bm \beta (y - \Pi_h^{k+1} y),  \nabla w_1)_{{\mathcal{T}_h}}  + (\nabla\cdot \bm \beta ( y - \Pi_h^{k+1} y), w_1)_{\mathcal T_h}\\
	&\quad +\langle (h^{-1}+\tau_1)(\Pi_h^{k+1} y -\mathcal I_h^{k+1}  y), w_1 - \mu_1 \rangle_{\partial{{\mathcal{T}_h}}\backslash\mathcal E_h^\partial}   \\
	&\quad +\langle (h^{-1}+\tau_1)(\Pi_h^{k+1} y - y), w_1 \rangle_{\mathcal E_h^\partial}\\
	&\quad  + \langle \bm{\beta}\cdot\bm n (\mathcal I_h^{k+1}  y-y), w_1-\mu_1 \rangle_{\partial{{\mathcal{T}_h}}\backslash \mathcal E_h^{\partial}} +\langle  (\bm \Pi_h^0 \bm q - \bm q)\cdot\bm n, w_1 - \mu_1\rangle_{\partial\mathcal T_h\backslash\mathcal E_h^\partial}\\
	&\quad  +\langle  (\bm \Pi_h^0 \bm q - \bm q)\cdot\bm n, w_1\rangle_{\mathcal  E_h^\partial}.
	\end{align*}
\end{proof}

\begin{lemma} \label{Pro_B2}
	For all $(\bm r_2, w_2,\mu_2)\in \bm V_h\times W_h\times M_h(o)$, we have
	\begin{align*}
	\hspace{1em}&\hspace{-1em}  \mathscr B_2 (\bm \Pi_h^k {\bm p},\Pi_h^{k+1} {z}, \mathcal I_h^{k+1}  z, \bm r_2, w_2, \mu_2)\\
	&= (y - y_d,w_2)_{\mathcal T_h} +\langle  \mathcal I_h^{k+1} z - z, \bm r_2\cdot \bm{n} \rangle_{\partial{{\mathcal{T}_h}}\backslash {\mathcal E_h^{\partial}}}  \\
	& \quad   -  (  \bm \beta (z - \Pi_h^{k+1} z),  \nabla w_2)_{{\mathcal{T}_h}}  +\langle (h^{-1}+\tau_2)(\Pi_h^{k+1} z -\mathcal I_h^{k+1}  z), w_2 - \mu_2 \rangle_{\partial{{\mathcal{T}_h}}\backslash\mathcal E_h^\partial}  \\
	&\quad  +\langle (h^{-1}+\tau_2)(\Pi_h^{k+1} z - z), w_2-\mu_2 \rangle_{\mathcal E_h^\partial}- \langle \bm{\beta}\cdot\bm n (\mathcal I_h^{k+1}  z-z), w_2 \rangle_{\partial{{\mathcal{T}_h}}\backslash \mathcal E_h^{\partial}} \\
	& \quad  +\langle  (\bm \Pi_h^k \bm p - \bm p) \cdot\bm n, w_2 - \mu_2\rangle_{\partial\mathcal T_h\backslash\mathcal E_h^\partial}  +\langle  (\bm \Pi_h^k \bm p - \bm p) \cdot\bm n, w_2\rangle_{\mathcal E_h^\partial}.
	\end{align*}
\end{lemma}
The proof proceeds in the same way as the proof of the above lemma.

Subtracting part 1 of the auxiliary problem \eqref{IEDG_u_a} from the equality in \Cref{Pro_B1} gives the following result:
\begin{lemma} \label{step1:error_equation}
	For $\varepsilon^{\bm q}_h={\bm\Pi}_h^0 \bm q-\bm q_h(u)$, $\varepsilon^{y}_h=\Pi_h^{k+1} y-y_h(u)$,  $\varepsilon^{\widehat y}_h=\mathcal I_h^{k+1} y- \widehat y_h^o(u)$,  we have
	\begin{align*}
	\hspace{1em}&\hspace{-1em}  \mathscr B_1 (\varepsilon_h^{\bm q},\varepsilon_h^{y},\varepsilon_h^{\widehat y}, \bm r_1, w_1, \mu_1)\\
	&= -\varepsilon^{-1}( {\bm q} - \bm \Pi_h^0 \bm q, \bm{r}_1)_{{\mathcal{T}_h}} +\langle  \mathcal I_h^{k+1} y - y, \bm r_1\cdot \bm{n} \rangle_{\partial{{\mathcal{T}_h}}\backslash {\mathcal E_h^{\partial}}} - (\bm \Pi_h^0 \bm q - \bm q,\nabla w_1)_{\mathcal{T}_h}\\
	&\quad +  (  \bm \beta (y - \Pi_h^{k+1} y),  \nabla w_1)_{{\mathcal{T}_h}}  + (\nabla\cdot \bm \beta ( y - \Pi_h^{k+1} y), w_1)_{\mathcal T_h}\\
	&\quad +\langle (h^{-1}+\tau_1)(\Pi_h^{k+1} y -\mathcal I_h^{k+1}  y), w_1 - \mu_1 \rangle_{\partial{{\mathcal{T}_h}}\backslash\mathcal E_h^\partial}   \\
	&\quad +\langle (h^{-1}+\tau_1)(\Pi_h^{k+1} y - y), w_1-\mu_1 \rangle_{\mathcal E_h^\partial}\\
	&\quad  + \langle \bm{\beta}\cdot\bm n (\mathcal I_h^{k+1}  y-y), w_1-\mu_1 \rangle_{\partial{{\mathcal{T}_h}}\backslash \mathcal E_h^{\partial}} +\langle  (\bm \Pi_h^0 \bm q - \bm q)\cdot\bm n, w_1-\mu_1 \rangle_{\partial\mathcal T_h\backslash \mathcal E_h^{\partial}}\\
	&\quad +\langle  (\bm \Pi_h^0 \bm q - \bm q)\cdot\bm n, w_1\rangle_{\mathcal E_h^{\partial}}
	\end{align*}
	for all $(\bm r_1, w_1,\mu_1)\in \bm V_h\times W_h\times M_h(o)$.
\end{lemma}

\subsubsection{Step 2: Estimates for $\varepsilon_h^{ q}$.}
\label{subsec:proof_step2}

\begin{lemma} \label{energy_norm_q}
	For $(\varepsilon_h^{\bm q}, \varepsilon_h^y, \varepsilon_h^{\widehat{y}})$ defined in \Cref{step1:error_equation}, we have
	\begin{align*}
	\hspace{1em}&\hspace{-1em}\|\varepsilon_h^{\bm q}\|_{\mathcal{T}_h}+h^{-\frac{1}{2}}\|\varepsilon_h^y-\varepsilon_h^{\widehat{y}}\|_{\partial \mathcal{T}_h\backslash\mathcal E_h^\partial} +h^{-\frac{1}{2}}\|\varepsilon_h^y\|_{\mathcal E_h^\partial} \\
	&\le  C\|\bm q - \bm \Pi_h^0 \bm q\|_{\mathcal T_h} + Ch \|\nabla \cdot \bm q\|_{\mathcal T_h}\\
	&\quad + Ch^{-1/2} (\|\Pi_h^{k+1} y - y\|_{\partial \mathcal T_h} + \|\mathcal I_h^{k+1} y - y\|_{\partial \mathcal T_h}).
	\end{align*}
\end{lemma}
\begin{proof}
	First, we take $(\bm r_1,w_1,\mu_1)=( \nabla \varepsilon_h^y,0,0)$ in  \Cref{step1:error_equation}, and by the definition of $\mathscr  B_1$ in \eqref{def_B1} we have 
	\begin{align}\label{nablay}
	\begin{split}
	\|\nabla \varepsilon_h^y\|_{\mathcal T_h}  &\le C  (\|\varepsilon^{\bm q}_h\|_{\mathcal T_h}+h^{-1/2}\|\varepsilon^y_h-\varepsilon^{\widehat y}_h\|_{\partial\mathcal T_h\backslash\mathcal E_h^\partial} +h^{-1/2}\|\varepsilon^y_h\|_{\mathcal E_h^\partial}) \\
	&\quad + Ch^{-1/2} \|\mathcal I_h^{k+1} y - y\|_{\partial \mathcal T_h} + C\|\bm q - \bm{\Pi}_h^0 \bm q\|_{\mathcal T_h}.
	\end{split}
	\end{align}
	
	Next,  taking $(\bm r_1,w_1,\mu_1)=( \varepsilon_h^{\bm q},\varepsilon_h^y,\varepsilon_h^{\widehat{y}})$ in  \Cref{property_B}, we get 
	\begin{align*}
	\hspace{1em}&\hspace{-1em}\mathscr B_1(\varepsilon^{\bm q}_h,\varepsilon^y_h,\varepsilon^{\widehat{y}}_h;\varepsilon^{\bm q}_h,\varepsilon^y_h,\varepsilon^{\widehat{y}}_h) \\
	&=\varepsilon^{-1}(\varepsilon^{\bm q}_h,\varepsilon^{\bm q}_h)_{\mathcal T_h}+ \langle (h^{-1}+\tau_1 - \frac 12 \bm \beta\cdot\bm n)(\varepsilon^{y}_h-\varepsilon^{\widehat{y}}_h),\varepsilon^y_h-\varepsilon^{\widehat{y}}_h\rangle_{\partial\mathcal T_h\backslash \mathcal E_h^\partial}\\
	&\quad-\frac 1 2(\nabla\cdot\bm\beta \varepsilon^y_h,\varepsilon^y_h)_{\mathcal T_h} +\langle (h^{-1}+\tau_1-\frac12\bm \beta\cdot\bm n) \varepsilon^y_h,\varepsilon^y_h\rangle_{\mathcal E_h^\partial}.
	\end{align*}
	
	On the other hand, take $(\bm r_1,w_1,\mu_1)=( \varepsilon_h^{\bm q},\varepsilon_h^y,\varepsilon_h^{\widehat{y}})$ in  \Cref{step1:error_equation} to obtain
	\begin{align*}
	\hspace{1em}&\hspace{-1em} \mathscr B_1(\varepsilon^{\bm q}_h,\varepsilon^y_h,\varepsilon^{\widehat{y}}_h;\varepsilon^{\bm q}_h,\varepsilon^y_h,\varepsilon^{\widehat{y}}_h)  \\
	&=-\varepsilon^{-1}( {\bm q} - \bm \Pi_h^0 \bm q, \varepsilon^{\bm q}_h)_{{\mathcal{T}_h}}   +\langle  \mathcal I_h^{k+1} y - y, \varepsilon^{\bm q}_h\cdot \bm{n} \rangle_{\partial{{\mathcal{T}_h}}\backslash \mathcal E_h^\partial} - (\bm \Pi_h^0 \bm q - \bm q,\nabla\varepsilon_h^y)_{\mathcal{T}_h} \\
	&\quad+(\bm \beta (y - \Pi_h^{k+1} y),  \nabla \varepsilon_h^y)_{{\mathcal{T}_h}} + (\nabla\cdot \bm \beta ( y - \Pi_h^{k+1} y), \varepsilon_h^y)_{\mathcal T_h}\\
	&\quad +\langle (h^{-1}+\tau_1)(\Pi_h^{k+1} y -\mathcal I_h^{k+1}  y), \varepsilon_h^y  - \varepsilon_h^{\widehat y}  \rangle_{\partial{{\mathcal{T}_h}}\backslash\mathcal E_h^\partial} \\
	&\quad  +\langle (h^{-1}+\tau_1)(\Pi_h^{k+1} y - y), \varepsilon_h^y \rangle_{\mathcal E_h^\partial} + \langle \bm{\beta}\cdot\bm n (\mathcal I_h^{k+1}  y-y), \varepsilon_h^y  - \varepsilon_h^{\widehat y}  \rangle_{\partial{{\mathcal{T}_h}}\backslash \mathcal E_h^{\partial}}\\
	&\quad +\langle  (\bm \Pi_h^0 \bm q - \bm q)\cdot\bm n, \varepsilon_h^y  - \varepsilon_h^{\widehat y} \rangle_{\partial{{\mathcal{T}_h}}\backslash \mathcal E_h^{\partial}} +\langle  (\bm \Pi_h^0 \bm q - \bm q)\cdot\bm n, \varepsilon_h^y\rangle_{\mathcal E_h^{\partial}}  \\
	& = T_1 + T_2 + T_3 + T_4 +T_5 + T_6 + T_7 + T_8 + T_9 + T_{10}.
	\end{align*}
	For the term $T_1$, the Cauchy-Schwarz inequality gives
	\begin{align*}
	T_1&  = -\varepsilon^{-1}( {\bm q} - \bm \Pi_h^0 \bm q, \varepsilon^{\bm q}_h)_{{\mathcal{T}_h}} \le C\|{\bm q} - \bm \Pi_h^0 \bm q\|_{{\mathcal{T}_h}}^2 + \frac{1}{16\varepsilon}\| \varepsilon_h^{\bm q}\|_{\mathcal{T}_h}^2.
	\end{align*}
	For the term $T_2$,  the Cauchy-Schwarz inequality  and an inverse inequality give
	\begin{align*}
	T_2	&= \langle  \mathcal I_h^{k+1} y - y, \varepsilon_h^{\bm q} \cdot \bm{n} \rangle_{\partial{{\mathcal{T}_h}}\backslash {\mathcal E_h^{\partial}}}\\
	& \le C h^{-\frac{1}{2}} \| \mathcal I_h^{k+1} y - y  \|_{\partial \mathcal{T}_h}\| \varepsilon_h^{\bm q}\|_{\mathcal{T}_h} \\
	&\le Ch^{-1}\| \mathcal I_h^{k+1} y - y  \|_{\partial \mathcal{T}_h}^2 + \frac{1}{16\varepsilon}\| \varepsilon_h^{\bm q}\|_{\mathcal{T}_h}^2.
	\end{align*}
	For the terms $T_3$, $T_4$ and $T_5$,  apply \eqref{nablay} and Young's inequality  to obtain
	\begin{align*}
	T_3& = - (\bm \Pi_h^0 \bm q - \bm q,\nabla\varepsilon_h^y)_{\mathcal{T}_h}\\
	&\le C \|\bm \Pi_h^0 \bm q - \bm q\|_{\mathcal T_h}^2 + \frac {1}{16\varepsilon}
	\|\varepsilon_h^{\bm{q}}\|_{\mathcal T_h}^2 + \frac 1 {16 h} \|{\varepsilon_h^y-\varepsilon_h^{\widehat{y}}}\|_{\partial \mathcal T_h\backslash\mathcal E_h^\partial}^2 +  \frac 1 {16 h} \|\varepsilon_h^y\|_{\mathcal E_h^\partial}^2,\\
	&\quad +Ch^{-1} \|\mathcal I_h^{k+1} y - y\|_{\partial \mathcal T_h}^2,\\
	T_4 &=  ( \bm \beta (y - \Pi_h^{k+1} y), \nabla \varepsilon_h^y)_{{\mathcal{T}_h}} \\
	&\le C \|y - \Pi_h^{k+1} y\|_{\mathcal T_h}^2 + \frac {1}{16\varepsilon}
	\|\varepsilon_h^{\bm{q}}\|_{\mathcal T_h}^2  + \frac 1 {16 h} \|{\varepsilon_h^y-\varepsilon_h^{\widehat{y}}}\|_{\partial \mathcal T_h\backslash\mathcal E_h^\partial}^2 +  \frac 1 {16 h} \|\varepsilon_h^y\|_{\mathcal E_h^\partial}^2,\\
	&\quad +Ch^{-1} \|\mathcal I_h^{k+1} y - y\|_{\partial \mathcal T_h}^2,\\
	T_5 &= (\nabla\cdot\bm{\beta}(y - \Pi_h^{k+1} y),\varepsilon_h^y)_{\mathcal T_h} \le C \|y - \Pi _h^{k+1}y\|_{\mathcal T_h}^2 +\norm{\sqrt{- \frac 1 2 \nabla\cdot\bm{\beta}}\varepsilon_h^y}_{\mathcal T_h}^2.
	\end{align*}
	For the terms $T_6$, $T_7$ and $T_8$, Young's equality gives
	\begin{align*}
	\hspace{1em}&\hspace{-1em} T_6 + T_7+T_8 \\
	&\le Ch^{-1} (\|\Pi_h^{k+1} y - y\|_{\partial \mathcal T_h}^2 + \|\mathcal I_h^{k+1} y - y\|_{\partial \mathcal T_h}^2) + \frac{1}{16h}\|{\varepsilon_h^y-\varepsilon_h^{\widehat{y}}}\|_{\partial \mathcal T_h\backslash\mathcal E_h^\partial}^2\\
	&\quad +  \frac 1 {16 h} \|\varepsilon_h^y\|_{\mathcal E_h^\partial}^2.
	\end{align*}
	For the last two terms $T_9$ and $T_{10}$, apply the trace inequality \Cref{improved_trace_inequality} to get 
	\begin{align*}
	T_9+ T_{10} &= \langle  (\bm \Pi_h^0 \bm q - \bm q)\cdot\bm n,{\varepsilon_h^y-\varepsilon_h^{\widehat{y}}}\rangle_{\partial\mathcal T_h\backslash\mathcal E_h^\partial} +  \langle  (\bm \Pi_h^0 \bm q - \bm q)\cdot\bm n,\varepsilon_h^y\rangle_{\mathcal E_h^\partial}\\
	&\le Ch\|\nabla\cdot((\bm \Pi_h^0 \bm q - \bm q)\|_{\mathcal T_h}^2 +Ch^{-1}\|\bm \Pi_h^0 \bm q - \bm q\|_{\mathcal T_h}^2+ \frac{1}{16h}\|{\varepsilon_h^y-\varepsilon_h^{\widehat{y}}}\|_{\partial \mathcal T_h\backslash\mathcal E_h^\partial}^2\\
	&\quad +Ch\|\bm \Pi_h^0 \bm q - \bm q\|_{\partial \mathcal T_h}^2 + \frac{1}{16h}\|\varepsilon_h^y\|_{\mathcal E_h^\partial}^2\\
	&= Ch\|\nabla\cdot \bm q\|_{\mathcal T_h}^2 +Ch^{-1}\|\bm \Pi_h^0 \bm q - \bm q\|_{\mathcal T_h}^2+ \frac{1}{16h}\|{\varepsilon_h^y-\varepsilon_h^{\widehat{y}}}\|_{\partial \mathcal T_h\backslash\mathcal E_h^\partial}^2\\
	&\quad +Ch\|\bm \Pi_h^0 \bm q - \bm q\|_{\partial \mathcal T_h}^2 + \frac{1}{16h}\|\varepsilon_h^y\|_{\mathcal E_h^\partial}^2.
	\end{align*}
	
	Summing the estimates for $\{T_i\}_{i=1}^{10}$ gives the result.
\end{proof}

\subsubsection{Step 3: Estimates for $\varepsilon_h^y$ by a duality argument.} \label{subsec:proof_step_3}
Next, we introduce the dual problem for any given $\Theta$ in $L^2(\Omega):$
\begin{equation}\label{Dual_PDE}
\begin{split}
\bm\Phi+\nabla\Psi&=0\qquad~\text{in}\ \Omega,\\
\nabla\cdot\bm{\Phi}  - \nabla\cdot(\bm{\beta} \Psi)&=\Theta \qquad\text{in}\ \Omega,\\
\Psi&=0\qquad~\text{on}\ \partial\Omega.
\end{split}
\end{equation}
Since the domain $\Omega$ is convex, we have the following regularity estimate
\begin{align}\label{dual_esti}
\norm{\bm \Phi}_{1} + \norm{\Psi}_{2} \le C_{\text{reg}} \norm{\Theta}_{\mathcal T_h},
\end{align}
\begin{lemma} \label{dual_y}
	For $\varepsilon_h^y$ defined in \Cref{step1:error_equation}, we have 
	\begin{align*}
	\|\varepsilon_h^y\|_{\mathcal T_h} &\le  C (\|\Pi_h^{k+1} y - y\|_{\mathcal T_h} + h^{1/2}(\|\Pi_h^{k+1} y - y\|_{\partial \mathcal T_h} +\|\mathcal I_h^{k+1} y - y\|_{\partial \mathcal T_h} ))\\
	&\quad  + C(h^2 \|\nabla\cdot\bm q\|_{\mathcal T_h} + h\| \bm q - \bm{\Pi}_h^0 \bm q\|_{\mathcal T_h}).
	\end{align*}
\end{lemma}
\begin{proof}
	First,   take $(\bm r_1,w_1,\mu_1)=(\bm \Pi_h^k \bm \Phi, -\Pi_h^{k+1}  \Psi, -\mathcal I_h^{k+1} \Psi)$ in  \Cref{step1:error_equation}  and use  $\Psi = 0$ on $\mathcal E_h^\partial$ to obtain
	\begin{align*}
	\hspace{1em}&\hspace{-1em}\mathscr B_1(\varepsilon^{\bm q}_h,\varepsilon^y_h,\varepsilon^{\widehat{y}}_h;\bm \Pi_h^k \bm \Phi,-\Pi_h^{k+1} \Psi, -\mathcal I_h^{k+1} \Psi)\\
	& =-\varepsilon^{-1}( {\bm q} - \bm \Pi_h^0 \bm q, \bm \Pi_h^k \bm \Phi)_{{\mathcal{T}_h}}  + \langle  \mathcal I_h^{k+1} y - y, \bm \Pi_h^k \bm \Phi \cdot \bm{n} \rangle_{\partial{{\mathcal{T}_h}}\backslash {\mathcal E_h^{\partial}}} \\
	& \quad + (\bm \Pi_h^0 \bm q - \bm q, \nabla \Pi_h^{k+1}  \Psi)_{\mathcal{T}_h}-  (\bm \beta (y - \Pi_h^{k+1} y),  \nabla \Pi_h^{k+1} \Psi)_{{\mathcal{T}_h}} \\
	&\quad   - (\nabla\cdot \bm \beta ( y - \Pi_h^{k+1} y), \Pi_h^{k+1} \Psi)_{\mathcal T_h}\\
	&\quad  -\langle (h^{-1}+\tau_1)(\Pi_h^{k+1} y -\mathcal I_h^{k+1}  y), \Pi_h^{k+1} \Psi - \mathcal I_h^{k+1} \Psi \rangle_{\partial{{\mathcal{T}_h}}\backslash\mathcal E_h^\partial} \\
	& \quad   -\langle (h^{-1}+\tau_1)(\Pi_h^{k+1} y - y), \Pi_h^{k+1} \Psi-\mathcal I_h^{k+1} \Psi \rangle_{\mathcal E_h^\partial} \\
	&\quad - \langle \bm{\beta}\cdot\bm n (\mathcal I_h^{k+1}  y-y), \Pi_h^{k+1} \Psi-\mathcal I_h^{k+1} \Psi \rangle_{\partial{{\mathcal{T}_h}}\backslash \mathcal E_h^{\partial}}\\
	&\quad -\langle  (\bm \Pi_h^0 \bm q - \bm q)\cdot\bm n,\Pi_h^{k+1} \Psi-\mathcal I_h^{k+1} \Psi\rangle_{\partial\mathcal T_h}.
	\end{align*}
	
	On the other hand, \Cref{identical_equa} and \Cref{Pro_B2} imply
	\begin{align*}
	\hspace{1em}&\hspace{-1em}\mathscr B_1(\varepsilon^{\bm q}_h,\varepsilon^y_h,\varepsilon^{\widehat{y}}_h;\bm \Pi_h^{k} \bm \Phi,-\Pi_h^{k+1} \Psi, -\mathcal I_h^{k+1} \Psi)\\
	& = -\mathscr B_2(\bm \Pi_h^k \bm \Phi,\Pi_h^{k+1} \Psi, \mathcal I_h^{k+1} \Psi; -\varepsilon^{\bm q}_h,\varepsilon^y_h,\varepsilon^{\widehat{y}}_h)\\
	& =  -(\Theta,\varepsilon^y_h)_{\mathcal T_h} +\langle  \mathcal I_h^{k+1} \Psi - \Psi, \varepsilon^{\bm q}_h \cdot \bm{n} \rangle_{\partial{{\mathcal{T}_h}}\backslash {\mathcal E_h^{\partial}}}  + (\bm \beta (\Psi - \Pi_h^{k+1} \Psi),  \nabla \varepsilon^y_h)_{{\mathcal{T}_h}}\\
	&\quad  -\langle (h^{-1}+\tau_2)(\Pi_h^{k+1} \Psi -\mathcal I_h^{k+1}  \Psi), \varepsilon^y_h - \varepsilon^{\widehat{y}}_h \rangle_{\partial{{\mathcal{T}_h}}\backslash\mathcal E_h^\partial}  \\
	&\quad  -\langle (h^{-1}+\tau_2)(\Pi_h^{k+1} \Psi - \Psi), \varepsilon^y_h \rangle_{\mathcal E_h^\partial}\\
	&\quad + \langle \bm{\beta}\cdot\bm n (\mathcal I_h^{k+1}  \Psi-\Psi), \varepsilon^y_h-\varepsilon^{\widehat{y}}_h \rangle_{\partial{{\mathcal{T}_h}}\backslash \mathcal E_h^{\partial}} \\
	& \quad  - \langle  (\bm \Pi_h^k \bm \Phi - \bm \Phi)\cdot\bm n,\varepsilon^y_h - \varepsilon^{\widehat{y}}_h\rangle_{\partial\mathcal T_h\backslash\mathcal E_h^{\partial}} -  \langle  (\bm \Pi_h^k \bm \Phi - \bm \Phi)\cdot\bm n,\varepsilon^y_h\rangle_{\mathcal E_h^{\partial}}.
	\end{align*}
	
	Compare the above two equalities and take $\Theta  = \varepsilon_h^y$ to obtain
	\begin{align*}
	\|\varepsilon_h^y\|_{\mathcal T_h}^2 & =   \langle  \mathcal I_h^{k+1} \Psi - \Psi, \varepsilon^{\bm q}_h \cdot \bm{n} \rangle_{\partial{{\mathcal{T}_h}}\backslash {\mathcal E_h^{\partial}}} +(  \bm \beta (\Psi - \Pi_h^{k+1} \Psi),  \nabla \varepsilon^y_h)_{{\mathcal{T}_h}}    \\
	&\quad  -\langle (h^{-1}+\tau_2)(\Pi_h^{k+1} \Psi -\mathcal I_h^{k+1}  \Psi), \varepsilon^y_h - \varepsilon^{\widehat{y}}_h \rangle_{\partial{{\mathcal{T}_h}}\backslash\mathcal E_h^\partial} \\
	& \quad -\langle (h^{-1}+\tau_2)(\Pi_h^{k+1} \Psi - \Psi), \varepsilon^y_h \rangle_{\mathcal E_h^\partial} \\
	&\quad  - \langle  (\bm \Pi_h^k \bm \Phi - \bm \Phi)\cdot\bm n,\varepsilon^y_h - \varepsilon^{\widehat{y}}_h\rangle_{\partial\mathcal T_h\backslash\mathcal E_h^{\partial}} -  \langle  (\bm \Pi_h^k \bm \Phi - \bm \Phi)\cdot\bm n,\varepsilon^y_h\rangle_{\mathcal E_h^{\partial}}\\
	&\quad + \langle \bm{\beta}\cdot\bm n (\mathcal I_h^{k+1}  \Psi-\Psi), \varepsilon^y_h-\varepsilon^{\widehat{y}}_h \rangle_{\partial{{\mathcal{T}_h}}\backslash \mathcal E_h^{\partial}}\\
	&\quad +\varepsilon^{-1}( {\bm q} - \bm \Pi_h^0 \bm q, \bm \Pi_h^k \bm \Phi)_{{\mathcal{T}_h}}  + \langle  \mathcal I_h^{k+1} y - y, \bm \Pi_h^k \bm \Phi \cdot \bm{n} \rangle_{\partial{{\mathcal{T}_h}}\backslash {\mathcal E_h^{\partial}}}\\
	& \quad - (\bm \Pi_h^{0} \bm q - \bm q, \nabla\Pi_h^{k+1}  \Psi)_{\mathcal{T}_h} -  (  \bm \beta (y - \Pi_h^{k+1} y),  \nabla \Pi_h^{k+1} \Psi)_{{\mathcal{T}_h}} \\
	&\quad  - (\nabla\cdot \bm \beta ( y - \Pi_h^{k+1} y), \Pi_h^{k+1} \Psi)_{\mathcal T_h}\\
	&\quad  -\langle (h^{-1}+\tau_1)(\Pi_h^{k+1} y -\mathcal I_h^{k+1}  y), \Pi_h^{k+1} \Psi - \mathcal I_h^{k+1} \Psi \rangle_{\partial{{\mathcal{T}_h}}\backslash\mathcal E_h^\partial} \\
	& \quad -\langle (h^{-1}+\tau_1)(\Pi_h^{k+1} y - y), \Pi_h^{k+1} \Psi-\mathcal I_h^{k+1} \Psi \rangle_{\mathcal E_h^\partial}\\
	&\quad  - \langle \bm{\beta}\cdot\bm n (\mathcal I_h^{k+1}  y-y), \Pi_h^{k+1} \Psi-\mathcal I_h^{k+1} \Psi \rangle_{\partial{{\mathcal{T}_h}}\backslash \mathcal E_h^{\partial}}\\
	&\quad -\langle  (\bm \Pi_h^0 \bm q - \bm q)\cdot\bm n,\Pi_h^{k+1} \Psi-\mathcal I_h^{k+1} \Psi\rangle_{\partial\mathcal T_h},\\
	& = \sum_{i=1}^{16} R_i.
	\end{align*}
	Estimates for the above $16$ terms can be easily obtained using the proof techniques in \Cref{energy_norm_q}; we omit the details.  We have 
	\begin{align*}
	\|\varepsilon_h^y\| &\le  C (\|\Pi_h^{k+1} y - y\|_{\mathcal T_h} + h^{1/2}(\|\Pi_h^{k+1} y - y\|_{\partial \mathcal T_h} +\|\mathcal I_h^{k+1} y - y\|_{\partial \mathcal T_h} ))\\
	&\quad  + C(h^2 \|\nabla\cdot\bm q\|_{\mathcal T_h} + h\| \bm q - \bm{\Pi}_h^0 \bm q\|_{\mathcal T_h}).
	\end{align*}
\end{proof}

As a consequence, a simple application of the triangle inequality gives optimal convergence rates for $\|\bm q -\bm q_h(u)\|_{\mathcal T_h}$ and $\|y -y_h(u)\|_{\mathcal T_h}$:

\begin{lemma}\label{lemma:step3_conv_rates}
	Let  $(\bm q, y)$ and $(\bm q_h(u), y_h(u))$ be the solutions of 
	\eqref{mixed} and \eqref{IEDG_u_a}, respectively. We have 
	\begin{subequations}
		\begin{align}
		\|\bm q -\bm q_h(u)\|_{\mathcal T_h} &\le \|\bm q - \bm{\Pi}_h^0 \bm q\|_{\mathcal T_h} + \|\varepsilon_h^{\bm q}\|_{\mathcal T_h}\nonumber \\
		&\le C \|\bm q - \bm \Pi_h^0 \bm q\|_{\mathcal T_h} + Ch \|\nabla \cdot \bm q\|_{\mathcal T_h}\nonumber\\
		&\quad + Ch^{-1/2} (\|\Pi_h^{k+1} y - y\|_{\partial \mathcal T_h} + \|\mathcal I_h^{k+1} y - y\|_{\partial \mathcal T_h}),\\
		\|y -y_h(u)\|_{\mathcal T_h} &\le \|y - \Pi_h^{k+1} y\|_{\mathcal T_h} + \|\varepsilon_h^{y}\|_{\mathcal T_h} \nonumber \\
		& \le C \|\Pi_h^{k+1} y - y\|_{\mathcal T_h} + Ch^2 \|\nabla\cdot\bm q\|_{\mathcal T_h} + Ch\| \bm q - \bm{\Pi}_h^0 \bm q\|_{\mathcal T_h}\nonumber\\
		&\quad +C h^{1/2}(\|\Pi_h^{k+1} y - y\|_{\partial \mathcal T_h} +\|\mathcal I_h^{k+1} y - y\|_{\partial \mathcal T_h}).
		\end{align}
	\end{subequations}
\end{lemma}

\subsubsection{Step 4: The error equation for part 2 of the auxiliary problem \eqref{IEDG_u_b}.} \label{subsec:proof_step_4}
Subtracting part 2 of the auxiliary problem \eqref{IEDG_u_b} from the  equality  in \Cref{Pro_B2} gives the error equation:
\begin{lemma} \label{step4:error_equation}
	For $\varepsilon^{\bm p}_h={\bm\Pi}_h^k\bm p-\bm p_h(u)$, $\varepsilon^{z}_h=\Pi_h^{k+1} z-z_h(u)$,  $\varepsilon^{\widehat z}_h=\mathcal I_h^{k+1} z- \widehat z_h^o(u)$,  we have
	\begin{align*}
	\hspace{1em}&\hspace{-1em}  \mathscr B_2 (\varepsilon_h^{\bm p},\varepsilon_h^{z},\varepsilon_h^{\widehat z}, \bm r_2, w_2, \mu_2)\\
	&= ( y - y_h(u), w_2)_{\mathcal T_h} -  (\bm \Pi \bm p - \bm p,\nabla w_2)_{\mathcal{T}_h}\\
	&\quad + \langle  \mathcal I_h^{k+1} z - z, \bm r_2\cdot \bm{n} \rangle_{\partial{{\mathcal{T}_h}}\backslash {\mathcal E_h^{\partial}}} - (  \bm \beta (z - \Pi_h^{k+1} z),  \nabla w_2)_{{\mathcal{T}_h}} \\
	&\quad +\langle (h^{-1}+\tau_2)(\Pi_h^{k+1} z -\mathcal I_h^{k+1}  z), w_2 - \mu_2 \rangle_{\partial{{\mathcal{T}_h}}\backslash\mathcal E_h^\partial} \\
	& \quad +\langle (h^{-1}+\tau_2)(\Pi_h^{k+1} z - z), w_2 \rangle_{\mathcal E_h^\partial} - \langle \bm{\beta}\cdot\bm n (\mathcal I_h^{k+1}  z-z), w_2-\mu_2 \rangle_{\partial{{\mathcal{T}_h}}\backslash \mathcal E_h^{\partial}}\\
	&\quad +\langle  (\bm \Pi_h^k \bm p - \bm p)\cdot\bm n, w_2 - \mu_2\rangle_{\partial{{\mathcal{T}_h}}\backslash \mathcal E_h^{\partial}} + \langle  (\bm \Pi_h^k \bm p - \bm p)\cdot\bm n, w_2 \rangle_{ \mathcal E_h^{\partial}} 
	\end{align*}
	for all $(\bm r_2, w_2,\mu_2)\in \bm V_h\times W_h\times M_h(o)$.
\end{lemma}

\subsubsection{Step 5: Estimates for $\varepsilon_h^p$ and $\varepsilon_h^z$ by an energy argument.} \label{subsec:proof_step_5}
\begin{lemma}\label{lemma:step5_main_lemma}
	For $(\varepsilon_h^{\bm p}, \varepsilon_h^z, \varepsilon_h^{\widehat{z}})$ defined in \Cref{step4:error_equation}, we have
	\begin{align*}
	\hspace{-1em}&\hspace{-1em}\|\varepsilon_h^{z}\|_{\mathcal T_h} +\|\varepsilon_h^{\bm p}\|_{\mathcal T_h}+h^{-\frac{1}{2}}\|\varepsilon_h^z-\varepsilon_h^{\widehat z}\|_{\partial\mathcal T_h\backslash \mathcal E_h^{\partial}} +h^{-\frac{1}{2}}\|\varepsilon_h^z\|_{\mathcal E_h^{\partial}}\\
	&\le C \|\Pi_h^{k+1} y - y\|_{\mathcal T_h} + Ch^{1/2}(\|\Pi_h^{k+1} y - y\|_{\partial \mathcal T_h} +\|\mathcal I_h^{k+1} y - y\|_{\partial \mathcal T_h} )\\
	&\quad  + Ch^2 \|\nabla\cdot\bm q\|_{\mathcal T_h} + Ch\| \bm q - \bm{\Pi}_h^0 \bm q\|_{\mathcal T_h}+   Ch^{1/2}\|\bm p - \bm \Pi_h^k \bm p\|_{\partial\mathcal T_h}\\
	&\quad +  C\|\bm p - \bm \Pi_h^k \bm p\|_{\mathcal T_h} +C h^{-1/2} (\|\Pi_h^{k+1} z - z\|_{\partial \mathcal T_h} + \|\mathcal I_h^{k+1} z- z\|_{\partial \mathcal T_h}).
	\end{align*}
\end{lemma}

\begin{proof}
	First,  we take $(\bm r_2,w_2,\mu_2)=( \varepsilon_h^{\bm p},\varepsilon_h^z,\varepsilon_h^{\widehat{z}})$ in  \Cref{property_B} to get 
	\begin{align*}
	\hspace{1em}&\hspace{-1em}\mathscr B_2(\varepsilon^{\bm p}_h,\varepsilon^z_h,\varepsilon^{\widehat{z}}_h;\varepsilon^{\bm p}_h,\varepsilon^z_h,\varepsilon^{\widehat{z}}_h) \\
	&=\varepsilon^{-1}(\varepsilon^{\bm p}_h,\varepsilon^{\bm p}_h)_{\mathcal T_h}+ \langle (h^{-1}+\tau_2 + \frac 12 \bm \beta\cdot\bm n)(\varepsilon^{z}_h-\varepsilon^{\widehat{z}}_h),\varepsilon^{z}_h-\varepsilon^{\widehat{z}}_h\rangle_{\partial\mathcal T_h\backslash \mathcal E_h^\partial}\\
	&\quad-\frac 1 2(\nabla\cdot\bm\beta \varepsilon^{z}_h,\varepsilon^{z}_h)_{\mathcal T_h} +\langle (h^{-1}+\tau_2+\frac12\bm \beta\cdot\bm n) \varepsilon^{z}_h,\varepsilon^{z}_h\rangle_{\mathcal E_h^\partial}.
	\end{align*}
	Next,  take $(\bm r_2,w_2,\mu_2)=( \varepsilon_h^{\bm p},\varepsilon_h^z,\varepsilon_h^{\widehat{z}})$ in  \Cref{step1:error_equation} to obtain
	\begin{align*}
	\hspace{1em}&\hspace{-1em} \mathscr B_2(\varepsilon^{\bm p}_h,\varepsilon^z_h,\varepsilon^{\widehat{z}}_h;\varepsilon^{\bm p}_h,\varepsilon^z_h,\varepsilon^{\widehat{z}}_h)  \\
	&= (\bm \Pi_h^{k} \bm p - \bm p,\nabla \varepsilon_h^z)_{\mathcal{T}_h}+\langle  \mathcal I_h^{k+1} z - z, \varepsilon^{\bm p}_h\cdot \bm{n} \rangle_{\partial{{\mathcal{T}_h}}} \\
	& \quad  -  (  \bm \beta (z - \Pi_h^{k+1} z),  \nabla \varepsilon_h^z)_{{\mathcal{T}_h}} +\langle (h^{-1}+\tau_2)(\Pi_h^{k+1} z -\mathcal I_h^{k+1}  z), \varepsilon_h^z  - \varepsilon_h^{\widehat z}  \rangle_{\partial{{\mathcal{T}_h}}\backslash\mathcal E_h^\partial}\\
	&\quad   +\langle (h^{-1}+\tau_2)(\Pi_h^{k+1} z - z), \varepsilon_h^z  \rangle_{\mathcal E_h^\partial}  - \langle \bm{\beta}\cdot\bm n (\mathcal I_h^{k+1}  z-z), \varepsilon_h^z  - \varepsilon_h^{\widehat z}  \rangle_{\partial{{\mathcal{T}_h}}\backslash \mathcal E_h^{\partial}}\\
	&\quad +\langle  (\bm \Pi_h^k \bm p - \bm p)\cdot\bm n, \varepsilon_h^z  - \varepsilon_h^{\widehat z}  \rangle_{\partial\mathcal T_h\backslash \mathcal E_h^{\partial}} +\langle  (\bm \Pi_h^k \bm p - \bm p)\cdot\bm n, \varepsilon_h^z \rangle_{\mathcal E_h^{\partial}} \\
	&\quad  +(y- y_h(u), \varepsilon_h^z)_{\mathcal T_h}\\ 
	& = T_1 + T_2 + T_3 + T_4 +T_5 + T_6 + T_7 + T_8.
	\end{align*}
	For the terms $T_1 - T_7$, follow the proof of \Cref{energy_norm_q} to get
	\begin{align*}
	\sum_{i=1}^7 T_i  &\le Ch^{-1} (\| \mathcal I_h z - z \|_{\partial \mathcal{T}_h}^2 + \|z - \Pi z\|_{\mathcal T_h}^2) + Ch\|\bm \Pi_h^{k} \bm p - \bm p\|_{\partial \mathcal T_h}^2 \\
	&\quad +C\|\bm \Pi_h^{k} \bm p - \bm p\|_{\mathcal T_h}^2 + \frac{1}{16\varepsilon}\| \varepsilon_h^{\bm p}\|_{\mathcal{T}_h}^2+ \frac 1 {16 h} \|{\varepsilon_h^z-\varepsilon_h^{\widehat{z}}}\|_{\partial \mathcal T_h\backslash\mathcal E_h^\partial}^2\\
	&\quad + \frac 1 {16 h} \|\varepsilon_h^z\|_{\mathcal E_h^\partial}^2 + h^{-1}\|\mathcal I_h^{k+1} z - z\|_{\partial \mathcal T_h}^2.
	\end{align*}
	For the last term $T_8$, we introduce a discrete Poincar{\'e} inequality from \cite{Brenner_Poincare_SINUM_2003}:
	\begin{equation}\label{poin_in}
	\begin{split}
	\|\varepsilon_h^z\|_{\mathcal T_h} &\le C(\|\nabla \varepsilon_h^z\|_{\mathcal T_h}+h^{-\frac 1 2} \|[\![\varepsilon_h^z]\!]\|_{\mathcal E_h})\\
	&= C(\|\nabla \varepsilon_h^z\|_{\mathcal T_h}+h^{-\frac 1 2} \|[\![\varepsilon_h^z - \varepsilon_h^{\widehat z}]\!]\|_{\mathcal E_h^o} +h^{-\frac 1 2} \| \varepsilon_h^z\|_{\mathcal E_h^\partial})\\
	&\le  C(\|\nabla \varepsilon_h^z\|_{\mathcal T_h}+h^{-\frac 1 2} \|\varepsilon_h^z - \varepsilon_h^{\widehat z}\|_{\partial \mathcal T_h\backslash \mathcal E_h^\partial} + h^{-\frac 1 2} \|\varepsilon_h^z\|_{\mathcal E_h^\partial}),
	\end{split}
	\end{equation}
	where $[\![\varepsilon_h^z]\!]_{\mathcal E_h^o}$ is the jump of $\varepsilon_h^z$ between adjacent elements and $[\![\varepsilon_h^z]\!]= \varepsilon_h^z$ on $\mathcal E_h^\partial$. Note that  the above equality in \eqref{poin_in} holds since $\varepsilon_h^{\widehat z} $ is single-valued on interior  faces, and the last inequality in  \eqref{poin_in} holds due to the triangle inequality.
	
	We note the inequality in \eqref{nablay} is valid with $(\bm p, z,\widehat z)$ in place of $(\bm q, y,\widehat y)$. This gives
	\begin{align*}
	T_8 \le C\| y - y_h(u)\|_{\mathcal T_h}^2 + \frac{1}{16\varepsilon}\| \varepsilon_h^{\bm p}\|_{\mathcal{T}_h}^2+ \frac 1 {16 h} \|{\varepsilon_h^z-\varepsilon_h^{\widehat{z}}}\|_{\partial \mathcal T_h\backslash \mathcal E_h^\partial}^2 + \frac 1 {16 h} \|\varepsilon_h^z\|_{\mathcal E_h^\partial}^2.
	\end{align*}
	Sum the above estimates and use \eqref{poin_in} to obtain the desired result.
\end{proof}

As a consequence, a simple application of the triangle inequality gives optimal convergence rates for $\|\bm p -\bm p_h(u)\|_{\mathcal T_h}$ and $\|z -z_h(u)\|_{\mathcal T_h}$:

\begin{lemma}\label{lemma:step5_conv_rates}
	Let $(\bm p, z)$ and $(\bm p_h(u), z_h(u))$ be the solutions of 
	\eqref{mixed} and \eqref{IEDG_u_b}, respectively. We have 
	\begin{align*}
	\hspace{1em}&\hspace{-1em}\|\bm p -\bm p_h(u)\|_{\mathcal T_h} + \|z -z_h(u)\|_{\mathcal T_h} \\
	&\le C \|\Pi_h^{k+1} y - y\|_{\mathcal T_h} + Ch^{1/2}(\|\Pi_h^{k+1} y - y\|_{\partial \mathcal T_h} +\|\mathcal I_h^{k+1} y - y\|_{\partial \mathcal T_h} )\\
	&\quad  + Ch^2 \|\nabla\cdot\bm q\|_{\mathcal T_h} + Ch\| \bm q - \bm{\Pi}_h^0 \bm q\|_{\mathcal T_h}+   Ch^{1/2}\|\bm p - \bm \Pi_h^k \bm p\|_{\partial\mathcal T_h}\\
	&\quad +  C\|\bm p - \bm \Pi_h^k \bm p\|_{\mathcal T_h} +C h^{-1/2} (\|\Pi_h^{k+1} z - z\|_{\partial \mathcal T_h} + \|\mathcal I_h^{k+1} z- z\|_{\partial \mathcal T_h}).
	\end{align*}
\end{lemma}

\subsubsection{Step 6: Estimate for $\|u-u_h\|_{\mathcal E_h^\partial}$, $\norm {y-y_h}_{\mathcal T_h}$ and $\norm {z-z_h}_{\mathcal T_h}$.}

Next, we bound the error between the solutions of the auxiliary problem \eqref{HDG_inter_u} and the discretization of the optimality system \eqref{IEDG_full_discrete}. This step and the next step  are very similar to Steps 6 and 7 in our previous works \cite{HuMateosSinglerZhangZhang1,HuMateosSinglerZhangZhang2}. We include these proofs here to make this paper self-contained.

For the remaining steps, we denote 
\begin{gather}\label{zetaqypz}
\begin{split}
\zeta_{\bm q} =\bm q_h(u)-\bm q_h, \  \zeta_{y} = y_h(u)-y_h,  \ \zeta_{\widehat y} = \widehat y_h^o(u)-\widehat y_h^o,\\
\zeta_{\bm p} =\bm p_h(u)-\bm p_h,\ \zeta_{z} = z_h(u)-z_h, \ \zeta_{\widehat z} = \widehat z_h^o(u)-\widehat z_h^o.
\end{split}
\end{gather}
Subtracting the auxiliary problem \eqref{HDG_inter_u} and the  system \eqref{IEDG_full_discrete}  gives the following error equations
\begin{subequations}\label{eq_yh}
	\begin{align}
	\mathscr B_1(\zeta_{\bm q},\zeta_y,\zeta_{\widehat y};\bm r_1, w_1,\mu_1) &= -\langle u-u_h, \bm{r}_1\cdot \bm{n} -( h^{-1}+\tau_1 - \bm \beta\cdot\bm n) w_1 \rangle_{{\mathcal E_h^{\partial}}}\label{eq_yh_yhu},\\
	\mathscr B_2(\zeta_{\bm p},\zeta_z,\zeta_{\widehat z};\bm r_2, w_2,\mu_2) &= (\zeta_y, w_2)_{\mathcal T_h}\label{eq_zh_zhu},
	\end{align}
\end{subequations}
for all $\left(\bm{r}_1, \bm{r}_2, w_1, w_2, \mu_1, \mu_2\right)\in \bm V_h\times\bm V_h\times W_h\times W_h\times {M}_h(o)\times {M}_h(o)$.
\begin{lemma}
	Let $(\bm p_h(u), z_h(u))$ be the solution of \eqref{HDG_inter_u}, $\zeta_y$ be defined as in \eqref{zetaqypz},  and $u$ and $u_h$ be the solutions of \eqref{mixed} and \eqref{IEDG_full_discrete}, respectively.  We have
	\begin{align*}
	\gamma \norm{u-u_h}_{\mathcal E_h^\partial}^2  + \norm {\zeta_y}_{\mathcal T_h}^2 &= \langle \gamma u+ \bm p_h(u)\cdot\bm n +(h^{-1}+\tau_2)   z_h(u),u-u_h\rangle_{\mathcal E_h^\partial}\\
	& \quad - \langle \gamma u_h+ \bm p_h\cdot\bm n +(h^{-1}+\tau_2)   z_h,u-u_h\rangle_{\mathcal E_h^\partial}.
	\end{align*}
\end{lemma}
\begin{proof}
	First, we have
	\begin{align*}
	& \langle \gamma u+\bm p_h(u)\cdot\bm n +( h^{-1}+ \tau_2) z_h(u),u-u_h\rangle_{\mathcal E_h^\partial}\\
	& \quad - \langle \gamma u_h+ \bm p_h\cdot\bm n + (h^{-1}+\tau_2) z_h,u-u_h\rangle_{\mathcal E_h^\partial}\\
	&= \gamma \norm{u-u_h}_{\mathcal E_h^\partial}^2 +  \langle \zeta_{\bm p}\cdot\bm n+(h^{-1}+\tau_2)  \zeta_z, u-u_h\rangle_{\mathcal E_h^\partial}.
	\end{align*}
	Next, \Cref{identical_equa} gives
	\begin{align*}
	\mathscr B_1(\zeta_{\bm q},\zeta_y,\zeta_{\widehat y};\zeta_{\bm p}, -\zeta_{z},-\zeta_{\widehat{ z}}) + \mathscr B_2(\zeta_{\bm p},\zeta_z,\zeta_{\widehat z}; -\zeta_{\bm q}, \zeta_{y},\zeta_{\widehat{y}}) = 0.
	\end{align*}
	One the other hand, we have
	\begin{align*}
	\hspace{1em}&\hspace{-1em}\mathscr B_1(\zeta_{\bm q},\zeta_y,\zeta_{\widehat y};\zeta_{\bm p}, -\zeta_{z},-\zeta_{\widehat{ z}}) + \mathscr B_2(\zeta_{\bm p},\zeta_z,\zeta_{\widehat z}; -\zeta_{\bm q}, \zeta_{y},\zeta_{\widehat{y}})\\
	&=  (\zeta_{ y},\zeta_{ y})_{\mathcal{T}_h} - \langle  u  -u_h, \zeta_{\bm p}\cdot \bm{n} + (h^{-1} + \tau_2)\zeta_z \rangle_{{\mathcal E_h^{\partial}}}.
	\end{align*}
	Comparing the above two equalities gives
	\begin{align*}
	(\zeta_{ y},\zeta_{ y})_{\mathcal{T}_h} = \langle u-u_h, \zeta_{\bm p}\cdot \bm{n} +(h^{-1}+\tau_2)  \zeta_z \rangle_{{\mathcal E_h^{\partial}}}.
	\end{align*}
\end{proof}

\begin{lemma}\label{erroruh}
	Let $(u,y)$ and $(u_h,y_h)$ be the solutions of \eqref{mixed} and \eqref{IEDG_full_discrete}, respectively. We have
	\begin{align*}
	\hspace{1em}&\hspace{-1em}\norm{u-u_h}_{\mathcal E_h^\partial} + \|y - y_h\|_{\mathcal T_h}\\
	&\le  C h^{-1/2}\|\Pi_h^{k+1} y - y\|_{\mathcal T_h} + C\|\Pi_h^{k+1} y - y\|_{\partial \mathcal T_h} +C\|\mathcal I_h^{k+1} y - y\|_{\partial \mathcal T_h} \\
	&\quad  + Ch^{3/2} \|\nabla\cdot\bm q\|_{\mathcal T_h} + Ch^{1/2}\| \bm q - \bm{\Pi}_h^0 \bm q\|_{\mathcal T_h}+ C h^{-1/2}\|\bm p - \bm \Pi_h^k \bm p\|_{\mathcal T_h} \\
	&\quad+Ch^{- 3/2}\|z-\Pi_h^{k+1} z\|_{\mathcal T_h} +C  \| \bm p - \bm \Pi_h^{k} \bm p\|_{\partial \mathcal T_h}\\
	&\quad + Ch^{-1} (\|\Pi_h^{k+1} z - z\|_{\partial \mathcal T_h} + \|\mathcal I_h^{k+1}z- z\|_{\partial \mathcal T_h}).
	\end{align*}
\end{lemma}

\begin{proof}
	%
	The optimality conditions yield $\gamma u+\bm p \cdot\bm n=0$ and $\gamma u_h+\bm p_h\cdot\bm n + h^{-1} z_h +\tau_2 z_h=0$ on $\mathcal E_h^{\partial}$.  Therefore, the above lemma gives
	\begin{align*}
	\gamma\norm{u-u_h}_{\mathcal E_h^\partial}^2  + \norm {\zeta_y}_{\mathcal T_h}^2 &= \langle \gamma u+\bm p_h(u)\cdot\bm n + h^{-1} z_h(u) + \tau_2 z_h(u),u-u_h\rangle_{\mathcal E_h^\partial}\\
	&=\langle (\bm p_h(u)-\bm p)\cdot\bm n + h^{-1} z_h(u)+ \tau_2 z_h(u) ,u-u_h\rangle_{\mathcal E_h^\partial}.
	\end{align*}
	Since $z=0$ on $\mathcal E_h^{\partial}$, we have
	\begin{align*}
	\norm {\bm p_h(u)-\bm p}_{\partial \mathcal T_h} &\le \norm {\bm p_h(u)-\bm{\Pi}_h^k\bm p}_{\partial \mathcal T_h} +\norm {\bm{\Pi}_h^k\bm p - \bm p}_{\partial \mathcal T_h}\\
	& \le C  h^{-\frac 1 2}\norm {\varepsilon_h^{\bm p}}_{\mathcal T_h} +C\norm {\bm{\Pi}_h^k\bm p - \bm p}_{\partial \mathcal T_h},\\
	\|z_h(u)\|_{\mathcal E_h^\partial} &=\|z_h(u) -\Pi_h^{k+1} z +\Pi_h^{k+1} z - z  \|_{\mathcal E_h^\partial} \\
	&\le   \|\varepsilon_h^z \|_{\mathcal E_h^\partial} + \|\Pi_h^{k+1} z - z\|_{\partial\mathcal T_h}.
	\end{align*}
	\Cref{lemma:step5_main_lemma} implies
	\begin{align*}
	\hspace{1em}&\hspace{-1em}\norm{u-u_h}_{\mathcal E_h^\partial}  + \|\zeta_y\|_{\mathcal T_h}\\
	&\le  C h^{-1/2}\|\Pi_h^{k+1} y - y\|_{\mathcal T_h} + C\|\Pi_h^{k+1} y - y\|_{\partial \mathcal T_h} +C\|\mathcal I_h^{k+1} y - y\|_{\partial \mathcal T_h} \\
	&\quad  + Ch^{3/2} \|\nabla\cdot\bm q\|_{\mathcal T_h} + Ch^{1/2}\| \bm q - \bm{\Pi}_h^0 \bm q\|_{\mathcal T_h}+ C h^{-1/2}\|\bm p - \bm \Pi_h^k \bm p\|_{\mathcal T_h} \\
	&\quad+Ch^{- 3/2}\|z-\Pi_h^{k+1} z\|_{\mathcal T_h} +C  \| \bm p - \bm \Pi_h^{k} \bm p\|_{\partial \mathcal T_h}\\
	&\quad + Ch^{-1} (\|\Pi_h^{k+1} z - z\|_{\partial \mathcal T_h} + \|\mathcal I_h^{k+1}z- z\|_{\partial \mathcal T_h}).
	\end{align*}
	The triangle inequality and \Cref{lemma:step3_conv_rates} yield the desired result.
\end{proof}

\subsubsection{Step 7: Estimates for $\|\boldmath p-\boldmath p_h\|_{\mathcal T_h}$,  $\|z-z_h\|_{\mathcal T_h}$, and $\|\boldmath q - \boldmath q_h\|_{\mathcal T_h}$}
\begin{lemma}
	For $(\zeta_{ \bm p},\zeta_z)$ defined in \eqref{zetaqypz}, we have
	\begin{align*}
	\hspace{1em}&\hspace{-1em}
	\norm {\zeta_{\bm p}}_{\mathcal T_h}  +\|\zeta_z\|_{\mathcal T_h}\\
	&\le  C h^{-1/2}\|\Pi_h^{k+1} y - y\|_{\mathcal T_h} + C\|\Pi_h^{k+1} y - y\|_{\partial \mathcal T_h} +C\|\mathcal I_h^{k+1} y - y\|_{\partial \mathcal T_h} \\
	&\quad  + Ch^{3/2} \|\nabla\cdot\bm q\|_{\mathcal T_h} + Ch^{1/2}\| \bm q - \bm{\Pi}_h^0 \bm q\|_{\mathcal T_h}+ C h^{-1/2}\|\bm p - \bm \Pi_h^k \bm p\|_{\mathcal T_h} \\
	&\quad+Ch^{- 3/2}\|z-\Pi_h^{k+1} z\|_{\mathcal T_h} +C  \| \bm p - \bm \Pi_h^{k} \bm p\|_{\partial \mathcal T_h}\\
	&\quad + Ch^{-1} (\|\Pi_h^{k+1} z - z\|_{\partial \mathcal T_h} + \|\mathcal I_h^{k+1}z- z\|_{\partial \mathcal T_h}).
	\end{align*}
\end{lemma}
\begin{proof}
	By the energy identity for $ \mathscr B_2 $ in \Cref{property_B}, and the second error equation \eqref{eq_zh_zhu}, we have
	\begin{align*}
	\hspace{1em}&\hspace{-1em}\mathscr B_2(\zeta_{\bm p},\zeta_z,\zeta_{\widehat z};\zeta_{\bm p},\zeta_z,\zeta_{\widehat z})\\
	&=\varepsilon^{-1}(\zeta_{\bm p},\zeta_{\bm p})_{\mathcal T_h}+ \langle (h^{-1}+\tau_2 + \frac 12 \bm \beta\cdot\bm n)(\zeta_z-\zeta_{\widehat z}),\zeta_{ z}-\zeta_{\widehat z}\rangle_{\partial\mathcal T_h\backslash \mathcal E_h^\partial}\\
	&\quad-\frac 1 2(\nabla\cdot\bm\beta \zeta_z,\zeta_z)_{\mathcal T_h} +\langle (h^{-1}+\tau_2+\frac12\bm \beta\cdot\bm n) \zeta_z,\zeta_z\rangle_{\mathcal E_h^\partial}\\
	&=(\zeta_y,\zeta_z)_{\mathcal T_h}\\
	&\le \norm{\zeta_y}_{\mathcal T_h} \norm{\zeta_z}_{\mathcal T_h}\\
	&\le C  \norm{\zeta_y}_{\mathcal T_h} (\|\nabla \zeta_z\|_{\mathcal T_h} + h^{-\frac 1 2} \|\zeta_z - \zeta_{\widehat z}\|_{\partial\mathcal T_h\backslash \mathcal E_h^\partial} +  h^{-\frac 1 2} \|\zeta_z\|_{\mathcal E_h^\partial}) \\
	& \le C  \norm{\zeta_y}_{\mathcal T_h}
	(\|\zeta_{\bm p}\|_{\mathcal T_h}+ h^{-\frac 1 2} \|\zeta_z - \zeta_{\widehat z}\|_{\partial\mathcal T_h\backslash \mathcal E_h^\partial} +  h^{-\frac 1 2} \|\zeta_z\|_{\mathcal E_h^\partial}+h^{-\frac 1 2} \|\mathcal I_h^{k+1} z - z\|_{\partial \mathcal T_h}),
	\end{align*}
	where for the last two inequalities we used the discrete Poincar{\'e} inequality  \eqref{poin_in} and also the inequality \eqref{nablay}.  This gives
	\begin{align}\label{zetapzinequality}
	\begin{split}
	\hspace{1em}&\hspace{-1em} \norm {\zeta_{\bm p}}_{\mathcal T_h} +h^{-\frac1 2}\|\zeta_z-\zeta_{\widehat z}\|_{\partial\mathcal T_h\backslash \mathcal E_h^\partial} +h^{-\frac1 2}\|\zeta_z\|_{\mathcal E_h^\partial}\\
	&\le  C h^{-1/2}\|\Pi_h^{k+1} y - y\|_{\mathcal T_h} + C\|\Pi_h^{k+1} y - y\|_{\partial \mathcal T_h} +C\|\mathcal I_h^{k+1} y - y\|_{\partial \mathcal T_h} \\
	&\quad  + Ch^{3/2} \|\nabla\cdot\bm q\|_{\mathcal T_h} + Ch^{1/2}\| \bm q - \bm{\Pi}_h^0 \bm q\|_{\mathcal T_h}+ C h^{-1/2}\|\bm p - \bm \Pi_h^k \bm p\|_{\mathcal T_h} \\
	&\quad+Ch^{- 3/2}\|z-\Pi_h^{k+1} z\|_{\mathcal T_h} +C  \| \bm p - \bm \Pi_h^{k} \bm p\|_{\partial \mathcal T_h}\\
	&\quad + Ch^{-1} (\|\Pi_h^{k+1} z - z\|_{\partial \mathcal T_h} + \|\mathcal I_h^{k+1}z- z\|_{\partial \mathcal T_h}).
	\end{split}
	\end{align}
	Using the discrete Poincar{\'e} inequality \eqref{poin_in} and \eqref{nablay} again yield
	\begin{align*}
	\|\zeta_z\|_{\mathcal T_h} & \le C( \|\nabla \zeta_z\|_{\mathcal T_h}+ h^{-\frac1 2}\|\zeta_z-\zeta_{\widehat z}\|_{\partial\mathcal T_h\backslash \mathcal E_h^\partial} +h^{-\frac1 2}\|\zeta_z\|_{\mathcal E_h^\partial})\\
	&\le C(\|\zeta_{\bm p}\|_{\mathcal T_h}+ h^{-\frac 1 2} \|\zeta_z - \zeta_{\widehat z}\|_{\partial\mathcal T_h\backslash \mathcal E_h^\partial} +  h^{-\frac 1 2} \|\zeta_z\|_{\mathcal E_h^\partial}+h^{-\frac 1 2} \|\mathcal I_h^{k+1} z - z\|_{\partial \mathcal T_h}).
	\end{align*}
	Finally, combine \eqref{zetapzinequality} and the above inequality to  give the desired result.
\end{proof}
\begin{lemma}
	If  $k\geq 1$, then
	\begin{align*}
	\norm {\zeta_{\bm q}}_{\mathcal T_h}  &\le  C h^{-1}\|\Pi_h^{k+1} y - y\|_{\mathcal T_h} + C h^{-1/2}\|\Pi_h^{k+1} y - y\|_{\partial \mathcal T_h}  \\
	&\quad  +C h^{-1/2}\|\mathcal I_h^{k+1} y - y\|_{\partial \mathcal T_h}+ Ch \|\nabla\cdot\bm q\|_{\mathcal T_h} + C\| \bm q - \bm{\Pi}_h^0 \bm q\|_{\mathcal T_h}\\
	&\quad+ C h^{-1}\|\bm p - \bm \Pi_h^k \bm p\|_{\mathcal T_h} + Ch^{- 2}\|z-\Pi_h^{k+1} z\|_{\mathcal T_h} +C  h^{-1/2} \| \bm p - \bm \Pi_h^{k} \bm p\|_{\partial \mathcal T_h}\\
	&\quad + Ch^{-3/2} (\|\Pi_h^{k+1} z - z\|_{\partial \mathcal T_h} + \|\mathcal I_h^{k+1}z- z\|_{\partial \mathcal T_h}).
	\end{align*}
\end{lemma}
\begin{proof}
	\Cref{property_B} and the first error equation \eqref{eq_yh_yhu} give 
	\begin{align*}
	\hspace{1em}&\hspace{-1em}\mathscr B_1(\zeta_{\bm q},\zeta_y,\zeta_{\widehat y};\zeta_{\bm q},\zeta_y,\zeta_{\widehat y}) \\
	&=\varepsilon^{-1}(\zeta_{\bm q}, \zeta_{\bm q})_{{\mathcal{T}_h}}+\langle (h^{-1}+\tau_1 - \frac 12 \bm \beta \cdot\bm n) (\zeta_y-\zeta_{\widehat y}) , \zeta_y-\zeta_{\widehat y}\rangle_{\partial{{\mathcal{T}_h}}\backslash\mathcal E_h^\partial}\\
	&\quad -{\frac 1 2} (\nabla\cdot\bm{\beta} \zeta_y,\zeta_y)_{\mathcal T_h} +  \langle (h^{-1}+\tau_1 - \frac 12 \bm \beta \cdot\bm n) \zeta_y, \zeta_y \rangle_{\mathcal E_h^\partial} \\
	&= -\langle  u-u_h, \zeta_{\bm q}\cdot \bm{n} + (\bm{\beta}\cdot\bm n-h^{-1}-\tau_1) \zeta_y \rangle_{{\mathcal E_h^{\partial}}}\\
	&= -\langle  u-u_h, \zeta_{\bm q}\cdot \bm{n} - (h^{-1}+\tau_2) \zeta_y \rangle_{{\mathcal E_h^{\partial}}}\\
	&=-\langle u-u_h, \zeta_{\bm q}\cdot \bm{n} - (h^{-1}+\tau_2) \zeta_y \rangle_{{\mathcal E_h^{\partial}}}\\
	&	\le C \norm {u-u_h}_{\mathcal E_h^{\partial}} (\norm {\zeta_{\bm q}}_{\mathcal E_h^{\partial}} + h^{-1} \norm {\zeta_{y}}_{\mathcal E_h^{\partial}} )\\
	&\le C h^{-\frac 1 2}\norm {u-u_h}_{\mathcal E_h^{\partial}} (\norm {\zeta_{\bm q}}_{\mathcal T_h} + h^{-\frac 1 2} \norm {\zeta_{y}}_{\mathcal E_h^{\partial}}).
	\end{align*}
	This gives
	\begin{align*}
	\norm {\zeta_{\bm q}}_{\mathcal T_h} \le C  h^{-\frac 1 2}\norm {u-u_h}_{\mathcal E_h^{\partial}}.
	\end{align*}
	The desired result can be obtained by the above inequality and \Cref{erroruh}.
\end{proof}

The above lemma, the triangle inequality, \Cref{lemma:step3_conv_rates},  \Cref{lemma:step5_conv_rates},  the estimates in \eqref{classical_ine} and \Cref{erroruh} complete the proof of the main result, \Cref{main_res}.

\section{Numerical Experiments}
\label{sec:numerics}

We consider three  examples on a unit square domain $ \Omega = [0,1]\times[0,1]\subset\mathbb{R}^2 $,  and set $\gamma = 1$ and $\bm \beta = [-x_1^2\sin(x_2), \cos(x_1)e^{x_2}]$.  In examples 1 and 2, we computed the convergence rates without having an explicit solution of the optimality system. We numerically approximated the solution using a very fine mesh with $h=\sqrt{2} \times 2^{-9}$, and compared this reference solution against other solutions computed on meshes with larger $h$.

\begin{example}\label{example1}
	First, we test the high regularity case by setting $ f= 0$ and $y_d = 1$. The numerical  results are shown in \Cref{table_1,table_2,table_3,table_4}. Next, we test the low regularity case by setting $f=0$ and $y_d = (x_1^2+x_2^2)^{-1/3}$. The numerical  results are shown in \Cref{table_5,table_6,table_7,table_8}.
	
	The convergence rate for the control $u$ and the flux $\bm q$ in \Cref{example1}  match  our theoretical results when $k=1$, but are higher than our theoretical results for $k=0$. The convergence rates for other variables are higher than our theory. Similar phenomena was reported in \cite{HuMateosSinglerZhangZhang1,HuMateosSinglerZhangZhang2}. We also note that the numerically observed convergence rates are higher for IEDG fir $y$ and $z$ in the case $k=0$.

	\begin{table}
		\begin{center}
			\begin{tabular}{cccccc|c}
				\hline
				${h}/{\sqrt 2}$ &1/16& 1/32&1/64 &1/128 & 1/256 &\textup{EO} \\
				\hline
				$\|\bm q - \bm q_h\|_{\mathcal T_h}$ &3.344E-01   &2.642E-01   &1.847E-01   &1.191E-01  & 7.230E-02\\
				order&-&    0.34  & 0.52&   0.63   &0.72&-
				\\
				\hline
				$\|\bm p- \bm p_h\|_{\mathcal T_h}$&   1.056E-01   &6.199E-02   &3.277E-02   &1.667E-02   &8.368E-03
				\\
				order&-&    0.77  & 0.92   &0.97   &0.99&0.5
				\\
				\hline
				$\norm{{y}-{y}_h}_{\mathcal T_h}$&   1.203E-01   &6.647E-02   &3.492E-02   &1.768E-02   &8.679E-03
				\\
				order&-&    0.85   &0.93   &0.98   &1.02&0.5
				\\
				\hline
				$\norm{{z}-{z}_h}_{\mathcal T_h}$&   1.371E-02   &3.955E-03   &1.464E-03   &6.427E-04   &2.972E-04
				\\
				order&-&    1.79  &1.43  &1.18&1.11 &0.5
				\\
				\hline
				$\norm{{ u}-{ u}_h}_{\mathcal E_h^\partial}$&   3.924E-01   &2.567E-01   &1.481E-01   &7.930E-02   &4.023E-02
				\\
				order&-&   0.61  &0.79   &0.90  &0.98&0.5\\
				\hline
			\end{tabular}
		\end{center}
		\caption{\Cref{example1}, high regularity test, $k=0$ and EDG: Errors, observed convergence orders, and expected order (EO) for the control $u$, the state $y$,  the dual state $z$ and their fluxes  $\bm q$ and $\bm p$.}\label{table_1}
	\end{table}
	\begin{table}
		\begin{center}
			\begin{tabular}{cccccc|c}
				\hline
				${h}/{\sqrt 2}$ &1/2& 1/4&1/8 &1/16 & 1/32 &\textup{EO} \\
				\hline
				$\|\bm q - \bm q_h\|_{\mathcal T_h}$ & 2.862E-01   &2.051Ee-01   &1.4036E-01   &9.019E-02   &5.552E-02\\
				order&-&    0.48  & 0.55   &0.64   &0.70&-\\
				\hline
				$\|\bm p- \bm p_h\|_{\mathcal T_h}$&   8.072E-02   &4.827E-02   &2.701E-02 & 1.471E-02   &7.754E-03
				\\
				order&-& 0.74   &0.83   &0.87   &0.92& 0.5\\
				\hline
				$\norm{{y}-{y}_h}_{\mathcal T_h}$& 4.608E-02   &1.759E-02   &5.644E-03   &1.988E-03   &6.866E-04 \\
				order&-&    1.38   &1.64   &1.50   &1.53 & 0.5\\
				\hline
				$\norm{{z}-{z}_h}_{\mathcal T_h}$&      2.053E-02   &6.196E-03   &1.482E-03   &3.220E-04   &6.930E-05
				
				\\
				order&-&    1.72  &2.06   &2.20   &2.21&0.5
				\\
				\hline
				$\norm{{ u}-{ u}_h}_{\mathcal E_h^\partial}$&   1.183E-01   &6.745E-02   &3.904E-02   &2.184E-02   &1.179E-02
				\\
				order&-&   0.81  &0.78   &0.83  &0.88&0.5\\
				\hline
			\end{tabular}
		\end{center}
		\caption{\Cref{example1}, high regularity test,  $k=0$ and IEDG: Errors, observed convergence orders, and expected order (EO) for the control $u$, the state $y$,  the dual state $z$ and their fluxes  $\bm q$ and $\bm p$.}\label{table_2}
	\end{table}

	\begin{table}
		\begin{center}
			\begin{tabular}{cccccc|c}
				\hline
				${h}/{\sqrt 2}$  &1/2& 1/4&1/8 &1/16 & 1/32 &\textup{EO} \\
				\hline
				$\|\bm q - \bm q_h\|_{\mathcal T_h}$ &   1.887E-01   &1.056E-01  & 5.596E-02  & 2.869E-02   &1.446E-02
				\\
				order&-&   0.83   &0.91 &0.96   &0.99&1.0\\
				\hline
				$\|\bm p- \bm p_h\|_{\mathcal T_h}$&   2.714E-02   &9.111E-03   &2.737E-03   &7.822E-04   &2.176E-04
				\\
				order&-&    1.57  &1.73&   1.80&   1.84&1.5\\
				\hline
				$\norm{{y}-{y}_h}_{\mathcal T_h}$&   1.693E-02   &4.892E-03   &1.263E-03   &3.207E-04   &8.168E-05
				\\
				order&-&    1.79   &1.95& 1.97& 1.97
				& 1.5\\
				\hline
				$\norm{{z}-{z}_h}_{\mathcal T_h}$&   2.144E-03  & 3.460E-04   &5.120E-05  & 7.168E-06   &9.818E-07
				\\
				order&-&    2.63&2.75   &2.83   &2.86
				& 1.5\\
				\hline
				$\norm{{ u}-{ u}_h}_{\mathcal E_h^\partial}$&   8.742E-02   &3.528E-02   &1.332E-02   &4.856E-03   &1.730E-03
				\\
				order&-&    1.30  &1.40   &1.45   &1.48&1.5\\
				\hline
			\end{tabular}
		\end{center}
		\caption{\Cref{example1}, high regularity test, $k=1$ and EDG: Errors, observed convergence orders, and expected order (EO) for the control $u$, the state $y$,  the dual state $z$ and their fluxes  $\bm q$ and $\bm p$.}\label{table_3}
	\end{table}

	\begin{table}
		\begin{center}
			\begin{tabular}{cccccc|c}
				\hline
				${h}/{\sqrt 2}$ & 1/2& 1/4&1/8 &1/16 & 1/32 &\textup{EO} \\
				\hline
				$\|\bm q - \bm q_h\|_{\mathcal T_h}$ &   1.681E-01   &9.613E-02   &5.145E-02   &2.656E-02   &1.346E-02
				\\
				order&-&   0.80   &0.90   &0.95&   0.98& 1.0\\
				\hline
				$\|\bm p- \bm p_h\|_{\mathcal T_h}$&   2.602E-02   &8.400E-03   &2.514E-03   &7.230E-04   &2.026E-04
				\\
				order&-&    1.63& 1.74&1.79&   1.83& 1.5\\
				\hline
				$\norm{{y}-{y}_h}_{\mathcal T_h}$&   1.206E-02   &3.290E-03   &8.556E-04   &2.197E-04  & 5.589E-05
				\\
				order&-&    1.87&1.94&1.96&1.97
				& 1.5\\
				\hline
				$\norm{{z}-{z}_h}_{\mathcal T_h}$&   2.438E-03   &4.037E-04   &5.591E-05   &7.421E-06   &9.718E-07
				\\
				order&-&    2.59&2.85&  2.91&   2.93
				& 1.5\\
				\hline
				$\norm{{ u}-{ u}_h}_{\mathcal E_h^\partial}$&   4.711E-02   &1.883E-02   &7.101E-03   &2.612E-03   &9.466E-04
				\\
				order&-&   1.32&1.40  &1.44&1.46
				&1.5\\
				\hline
			\end{tabular}
		\end{center}
		\caption{\Cref{example1},  high regularity test, $k=1$ and IEDG: Errors, observed convergence orders, and expected order (EO) for the control $u$, the state $y$,  the dual state $z$ and their fluxes  $\bm q$ and $\bm p$.}\label{table_4}
	\end{table}

	\begin{table}
		\begin{center}
			\begin{tabular}{cccccc|c}
				\hline
				${h}/{\sqrt 2}$ &1/2& 1/4&1/8 &1/16 & 1/32 &\textup{EO} \\
				\hline
				$\|\bm q - \bm q_h\|_{\mathcal T_h}$ &   4.127E-01  & 3.277E-01   &2.512E-01   &1.909E-01   &1.440E-01
				\\
				order&-&   0.33   &0.38   &0.39   &0.40& -\\
				\hline
				$\|\bm p- \bm p_h\|_{\mathcal T_h}$&   1.398E-01   &8.175E-02   &4.388E-02   &2.275E-02   &1.159E-02
				\\
				order&-&    0.77   &0.89   &0.94 &0.97& 0.5\\
				\hline
				$\norm{{y}-{y}_h}_{\mathcal T_h}$&   1.545E-01   &8.605E-02   &4.654E-02   &2.429E-02   &1.223E-02
				\\
				order&-&   0.84  &0.88&0.94&0.99 & 0.5\\
				\hline
				$\norm{{z}-{z}_h}_{\mathcal T_h}$&   2.072E-02   &6.698E-03   &2.430E-03   &1.002E-03   &4.478E-04
				\\
				order&-&    1.62&1.46&   1.27&   1.16& 0.5\\
				\hline
				$\norm{{ u}-{ u}_h}_{\mathcal E_h^\partial}$&   5.161E-01   &3.236E-01   &1.888E-01   &1.056E-01   &5.706E-02
				\\
				order&-&    0.67  &0.77  &0.83  & 0.88&0.5\\
				\hline
			\end{tabular}
		\end{center}
		\caption{\Cref{example1}, low regularity test, $k=0$ and EDG: Errors, observed convergence orders, and expected order (EO) for the control $u$, the state $y$,  the dual state $z$ and their fluxes  $\bm q$ and $\bm p$.}\label{table_5}
	\end{table}
	\begin{table}
		\begin{center}
			\begin{tabular}{cccccc|c}
				\hline
				${h}/{\sqrt 2}$ &1/2& 1/4&1/8 &1/16 & 1/32&\textup{EO} \\
				\hline
				$\|\bm q - \bm q_h\|_{\mathcal T_h}$ &3.600E-01   &2.788E-01   &2.202E-01   &1.729E-01   &1.372E-01\\
				order&-&    0.36   &0.34 &0.34& 0.33&-
				\\
				\hline
				$\|\bm p- \bm p_h\|_{\mathcal T_h}$&   1.072E-01   &6.488E-02   &3.701E-02   &2.036E-02   &1.080E-02
				\\
				order&-&    0.72  &0.80   &0.86&0.91& 0.5\\
				\hline
				$\norm{{y}-{y}_h}_{\mathcal T_h}$&   5.983E-02   &2.280E-02  & 8.000E-03   &3.088E-03   &1.230E-03
				\\
				order&-&    1.39   &1.51  &1.37 &1.32& 0.5\\
				\hline
				$\norm{{z}-{z}_h}_{\mathcal T_h}$&   2.883E-02   &9.238E-03   &2.387E-03   &5.672E-04 &  1.314E-04
				\\
				order&-&    1.64  &1.95   &2.07  &2.10& 0.5\\
				\hline
				$\norm{{ u}-{ u}_h}_{\mathcal E_h^\partial}$&   1.511E-01   &8.549E-02   &5.284E-02   &3.198E-02   &1.911E-02
				\\
				order&-&    0.82  &0.69   &0.72   &0.74
				&0.5\\
				\hline
			\end{tabular}
		\end{center}
		\caption{\Cref{example1},  low regularity test, $k=0$ and IEDG: Errors, observed convergence orders, and expected order (EO) for the control $u$, the state $y$,  the dual state $z$ and their fluxes  $\bm q$ and $\bm p$.}\label{table_6}
	\end{table}

	\begin{table}
		\begin{center}
			\begin{tabular}{cccccc|c}
				\hline
				${h}/{\sqrt 2}$  &1/2& 1/4&1/8 &1/16 & 1/32 &\textup{EO} \\
				\hline
				$\|\bm q - \bm q_h\|_{\mathcal T_h}$ &   2.486E-01   &1.772E-01   &1.347E-01   &1.041E-01   &7.919E-02\\
				order&-&    0.48   &0.40   &0.37  &0.39
				& 0.33\\
				\hline
				$\|\bm p- \bm p_h\|_{\mathcal T_h}$&   3.467E-02   &1.305E-02   &4.971E-03   &1.941E-03   &7.691E-04
				\\
				order&-&    1.40&1.39   &1.35&1.33& 0.83\\
				\hline
				$\norm{{y}-{y}_h}_{\mathcal T_h}$&   1.837E-02   &5.972E-03   &2.199E-03   &8.997E-04   &3.726E-04
				\\
				order&-&    1.62   &1.44&1.28&1.27& 0.83\\
				\hline
				$\norm{{z}-{z}_h}_{\mathcal T_h}$&   5.359E-03   &1.333E-03   &3.007E-04   &6.417E-05   &1.333E-05
				\\
				order&-&    2.00   &2.14&   2.22&   2.26& 0.83\\
				\hline
				$\norm{{ u}-{ u}_h}_{\mathcal E_h^\partial}$&   9.307E-02   &4.230E-02   &2.254E-02   &1.288E-02   &7.352E-03
				\\
				order&-&    1.13   &0.90   &0.80  &0.80& 0.83
				\\
				\hline
			\end{tabular}
		\end{center}
		\caption{\Cref{example1},  low regularity test, $k=1$ and EDG: Errors, observed convergence orders, and expected order (EO) for the control $u$, the state $y$,  the dual state $z$ and their fluxes  $\bm q$ and $\bm p$.}\label{table_7}
	\end{table}

	\begin{table}
		\begin{center}
			\begin{tabular}{cccccc|c}
				\hline
				${h}/{\sqrt 2}$  &1/2& 1/4&1/8 &1/16 & 1/32 &\textup{EO} \\
				\hline
				$\|\bm q - \bm q_h\|_{\mathcal T_h}$ &   2.431E-01   &1.786E-01   &1.372E-01   &1.072E-01   &8.321E-02
				\\
				order&-&  0.44  &0.38   &0.35   &0.36& 0.33\\
				\hline
				$\|\bm p- \bm p_h\|_{\mathcal T_h}$&   3.411E-02   &1.270E-02   &4.868E-03   &1.905E-03   &7.546E-04
				\\
				order&-&    1.42   &1.38 &1.35&1.33
				& 0.83\\
				\hline
				$\norm{{y}-{y}_h}_{\mathcal T_h}$&   1.535E-02   &5.202E-03   &1.869E-03   &7.186E-04   &2.848E-04
				\\
				order&-&    1.56&1.47&1.37  & 1.33& 0.83\\
				\hline
				$\norm{{z}-{z}_h}_{\mathcal T_h}$&   5.255E-03   &1.277E-03   &2.825E-04   &5.940E-05   &1.212E-05
				\\
				order&-&    2.04  &2.17  &2.24  & 2.29& 0.83\\
				\hline
				$\norm{{ u}-{ u}_h}_{\mathcal E_h^\partial}$&   6.326E-02   &3.179E-02   &1.716E-02   &9.636E-03   &5.458E-03
				\\
				order&-&    1.00   &0.89 & 0.83  & 0.82
				&0.83\\
				\hline
			\end{tabular}
		\end{center}
		\caption{\Cref{example1},  low regularity test, $k=1$ and IEDG: Errors, observed convergence orders, and expected order (EO) for the control $u$, the state $y$,  the dual state $z$ and their fluxes  $\bm q$ and $\bm p$.}\label{table_8}
	\end{table}
\end{example}

\begin{example}
	Next, we demonstrate the performance of the EDG and IEDG methods in the convection dominated case. We do not  compute the  convergence rates here; instead for illustration we plot the state $y_h$ in \Cref{computed_state}. Moreover, we also plot the approximate state computed using the CG method. All computations are on the same mesh with $h = \sqrt{2} \times 2^{-8}$ and the data chosen as
	\begin{align*}
	\varepsilon = 10^{-6},\quad f = x_1x_2, \ \ \textup{and} \ \  y_d = 1. 
	\end{align*}
	\begin{figure}
		\centerline{
			\hbox{\includegraphics[width=0.31\textwidth]{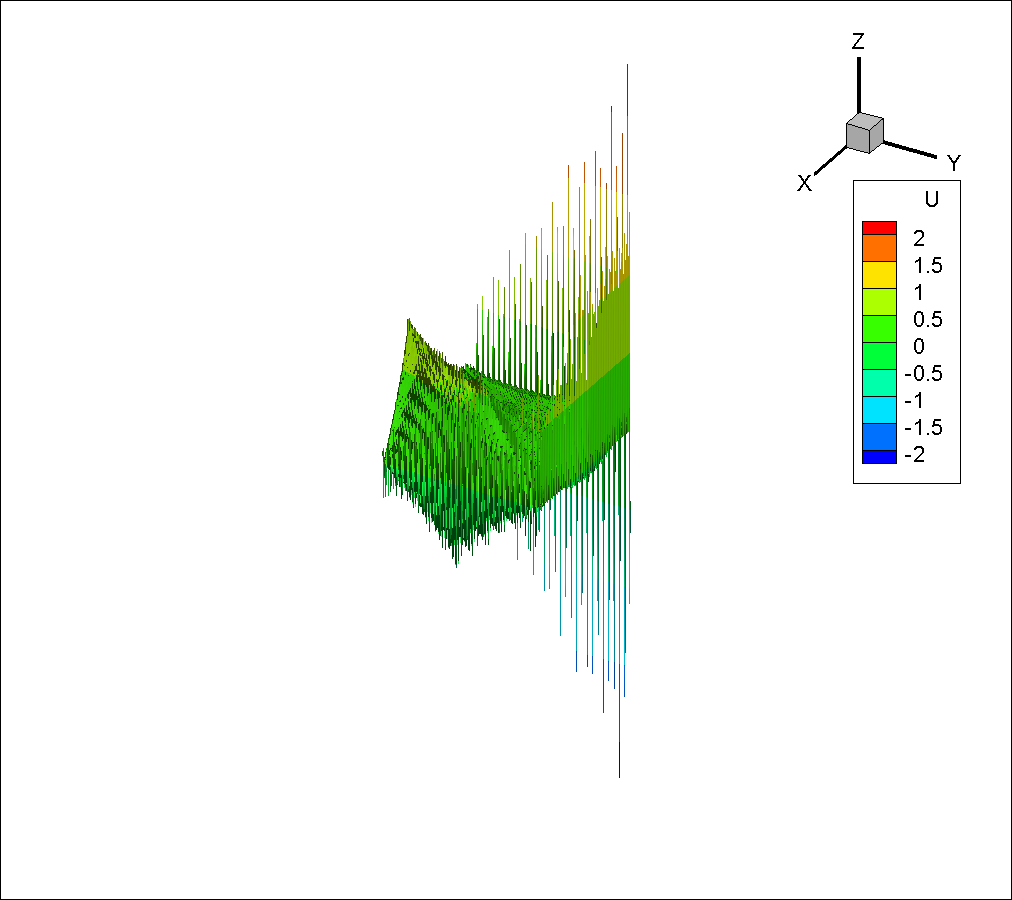}}
			\hbox{\includegraphics[width=0.31\textwidth]{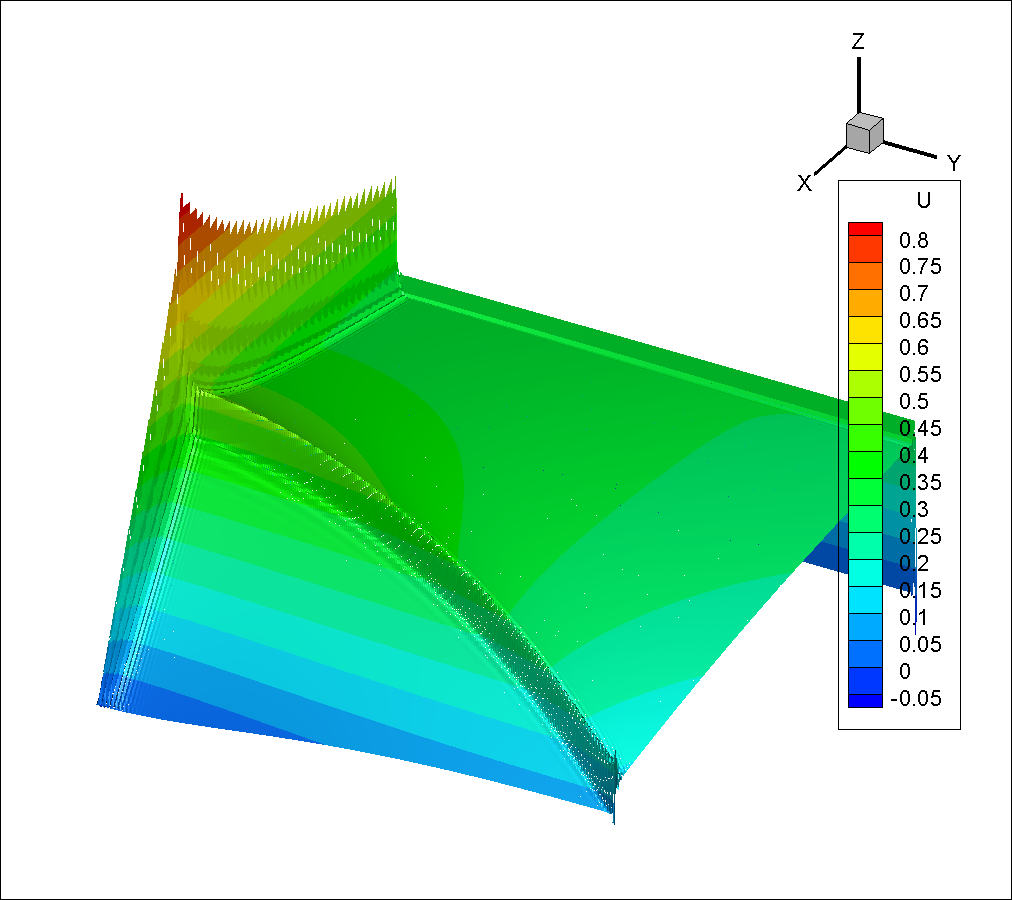}}
			\hbox{\includegraphics[width=0.31\textwidth]{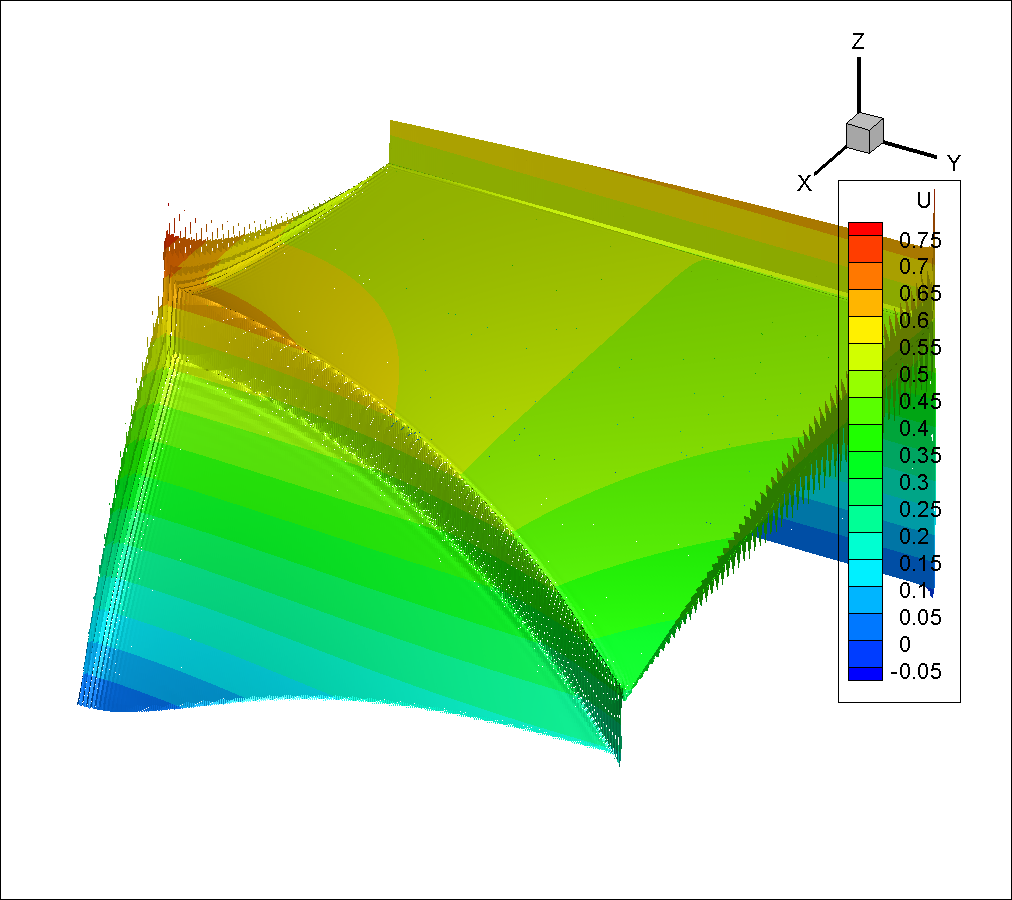}}
		}
		\caption{The computed state $y_h$ by CG (left), EDG (middle) and IEDG (right).}
		\label{computed_state}
		\centering
	\end{figure}
	
	We observe that the approximate state computed by the  CG method is highly oscillatory, but we only have a small oscillation near the sharp change with the  EDG and IEDG methods. Furthermore, the oscillations in the IEDG solutions are slightly smaller than in the EDG solution.
\end{example}

\section{Conclusion}

In this work, we approximate the solution of a convection diffusion Dirichlet boundary control problem by EDG and IEDG methods. We obtained an optimal convergence rate for the control for both high regularity and low regularity cases. Instead of introducing a special projection as in \cite{HuMateosSinglerZhangZhang2}, we used an improved trace inequality for the low regularity case. This simplified the analysis. Finally, some numerical experiments showed that the EDG and IEDG methods are suitable for convection dominated problems. It is worth mentioning that the number of degrees of freedom of EDG and IEDG methods are lower than the HDG method.

\bibliographystyle{plain}
\bibliography{/Users/yangwenzhang/Desktop/Research/Bib/Model_Order_Reduction,/Users/yangwenzhang/Desktop/Research/Bib/Dirichlet_Boundary_Control,/Users/yangwenzhang/Desktop/Research/Bib/Ensemble,/Users/yangwenzhang/Desktop/Research/Bib/HDG,/Users/yangwenzhang/Desktop/Research/Bib/EDG,/Users/yangwenzhang/Desktop/Research/Bib/Interpolatory,/Users/yangwenzhang/Desktop/Research/Bib/Mypapers,/Users/yangwenzhang/Desktop/Research/Bib/Added}

\end{document}